\documentclass[11pt]{amsart}
\usepackage[utf8]{inputenc}
\usepackage{enumerate,amssymb,xcolor,comment}
\usepackage[a4paper, footskip=1cm, headheight = 16pt, top=3.7cm, bottom=3cm,  right=2.5cm,  left=2.5cm ]{geometry}
\usepackage[colorlinks,linkcolor=blue,citecolor=red,hypertexnames=false]{hyperref}

\newtheorem{question}{Question}[section]
\newtheorem{lemma}[question]{Lemma}
\newtheorem{theorem}[question]{Theorem}
\newtheorem{conjecture}[question]{Conjecture}

\usepackage{xcolor}

\usepackage[foot]{amsaddr}

\usepackage[T1]{fontenc}

\usepackage[leqno]{amsmath}

\makeatletter
\newcommand{\leqnomode}{\tagsleft@true}
\newcommand{\reqnomode}{\tagsleft@false}
\makeatother

\usepackage{latexsym}
\usepackage{amsfonts}

\usepackage{tikz}
\usepackage{float}
\usepackage{lmodern}

\def\dd{\hbox{-}}

\usepackage{marvosym}               

\DeclareMathOperator{\tw}{tw}
\DeclareMathOperator{\width}{width}

\DeclareMathOperator{\Hub}{Hub}
\DeclareMathOperator{\hdim}{hdim}
\DeclareMathOperator{\Core}{Core}

\newcounter{tbox}

\newcommand{\sta}[1]{\vspace*{0.3cm} \refstepcounter{tbox}
  \noindent{ \parbox{\textwidth}{(\thetbox) \emph{#1}}}\vspace*{0.3cm}}

\newcommand{\mylongtitle}[1]{%
  \ifodd\value{page}%
    \protect\parbox{0.97\linewidth}{#1}\hfill%
  \else%
    \hfill\protect\parbox{0.97\linewidth}{#1}%
  \fi%
}

\makeatletter
\newcommand{\otherlabel}[2]{\protected@edef\@currentlabel{#2}\label{#1}}
\makeatother

\mathchardef\mh="2D

\title[Induced subgraphs and tree decompositions X.]{Induced subgraphs and tree decompositions\\
X. Towards logarithmic treewidth for even-hole-free graphs}

\author{Tara Abrishami$^{\ast \dagger}$}
\author{Bogdan Alecu$^{\ast \ast \mathparagraph}$}
\author{Maria Chudnovsky$^{\ast \dagger}$}
\author{Sepehr Hajebi $^{\mathsection}$}
\author{Sophie Spirkl$^{\mathsection \parallel}$}
\address{$^{\ast}$Princeton University, Princeton, NJ, USA}
\address{$^{\mathparagraph}$ Supported by DMS-EPSRC Grant EP/V002813/1.}
\address{$^{**}$School of Computing, University of Leeds, Leeds, UK}
\address{$^{\mathsection}$Department of Combinatorics and Optimization, University of Waterloo, Waterloo, Ontario, Canada}
\address{$^{\dagger}$ Supported by NSF-EPSRC Grant DMS-2120644 and by AFOSR grant FA9550-22-1-0083.}
\address{$^{\parallel}$ We acknowledge the support of the Natural Sciences and Engineering Research Council of Canada (NSERC), [funding reference number RGPIN-2020-03912].
Cette recherche a \'et\'e financ\'ee par le Conseil de recherches en sciences naturelles et en g\'enie du Canada (CRSNG), [num\'ero de r\'ef\'erence RGPIN-2020-03912]. This project was funded in part by the Government of Ontario.}

\date {\today}

\allowdisplaybreaks

\begin{document}

\maketitle
\begin{abstract}
  A generalized $t$-pyramid is  a graph obtained from a certain kind of tree
  (a subdivided star or a subdivided cubic caterpillar) and  the line graph of a
  subdivided cubic caterpillar by identifying simplicial vertices.
  We prove that for every integer $t$ there exists a constant $c(t)$ such that
  every $n$-vertex even-hole-free graph with no clique of size $t$ and no
  induced subgraph isomorphic to a generalized $t$-pyramid has treewidth at most $c(t)\log{n}$. This settles a
  special case of a conjecture of  Sintiari and Trotignon; this bound is also best possible for the class.
  It follows that several \textsf{NP}-hard problems such as \textsc{Stable Set}, \textsc{Vertex Cover}, \textsc{Dominating Set} and \textsc{Coloring} admit polynomial-time algorithms on this class of graphs.
  Results from this paper are also used in later papers of the series, in particular to solve the full version of the Sintiari-Trotignon conjecture.
\end{abstract}

\section{Introduction} \label{intro}
All graphs in this paper are finite and simple. Let $G = (V(G),E(G))$ be a graph. For a set $X \subseteq V(G)$, we denote by $G[X]$ the subgraph of $G$ induced by $X$. For $X \subseteq V(G)$, $G \setminus X$ denotes the subgraph induced by $V(G) \setminus X$. In this paper, we use induced subgraphs and their vertex sets interchangeably. In particular, for $X, Y \subseteq V(G)$ we use the notation $X \cup Y$ to mean the graph $G[X \cup Y]$ as well as the set  $X \cup Y$.
We never use $\cup$ or $\bigcup$ to denote the disjoint union of two graphs.
For a vertex $v \in V(G)$, we often write $v$ to mean $\{v\}$.

For graphs $G$ and $H$, we say that $G$ {\em contains} $H$ if some induced subgraph of $G$ is isomorphic to $H$. For a family $\mathcal{H}$ of graphs, $G$
{\em contains}
$\mathcal{H}$ if $G$ contains a member of $\mathcal{H}$. Finally,
$G$ is {\em $\mathcal{H}$-free} if $G$ does not contain $\mathcal{H}$.

Let $v \in V(G)$. The \emph{open neighborhood of $v$}, denoted by $N(v)$, is the set of all vertices in $V(G)$ adjacent to $v$. The \emph{closed neighborhood of $v$}, denoted by $N[v]$, is $N(v) \cup \{v\}$. Let $X \subseteq V(G)$. The \emph{open neighborhood of $X$}, denoted by $N(X)$, is the set of all vertices in $V(G) \setminus X$ with at least one neighbor in $X$. The \emph{closed neighborhood of $X$}, denoted by $N[X]$, is $N(X) \cup X$. If $H$ is an induced subgraph of $G$ and $X \subseteq V(G)$, then $N_H(X)=N(X) \cap H$ and $N_H[X]=N_H(X) \cup (X \cap H)$. Let $Y \subseteq V(G)$ be disjoint from $X$. We say $X$ is \textit{complete} to $Y$ if all possible edges with an end in $X$ and an end in $Y$ are present in $G$, and $X$ is \emph{anticomplete}
to $Y$ if there are no edges between $X$ and $Y$.

For a graph $G = (V(G),E(G))$, a \emph{tree decomposition} $(T, \chi)$ of $G$ consists of a tree $T$ and a map $\chi: V(T) \to 2^{V(G)}$ with the following properties: 
\begin{enumerate}[(i)]
\itemsep -.2em
    \item For every $v \in V(G)$, there exists $t \in V(T)$ such that $v \in \chi(t)$. 
    
    \item For every $v_1v_2 \in E(G)$, there exists $t \in V(T)$ such that $v_1, v_2 \in \chi(t)$.
    
    \item For every $v \in V(G)$, the subgraph of $T$ induced by $\{t \in V(T) \mid v \in \chi(t)\}$ is connected.
\end{enumerate}
 
For each $t\in V(T)$, we refer to $\chi(t)$ as a \textit{bag of} $(T, \chi)$.  The \emph{width} of a tree decomposition $(T, \chi)$, denoted by $\width(T, \chi)$, is $\max_{t \in V(T)} |\chi(t)|-1$. The \emph{treewidth} of $G$, denoted by $\tw(G)$, is the minimum width of a tree decomposition of $G$. The term ``treewidth'' and the study of the structure of graphs with large treewidth were introduced by Robertson and Seymour \cite{RS-GMII} as part of the Graph Minors series.

A {\em hole} in a graph is an induced cycle with at least four vertices. The
{\em length} of a hole is the number of vertices in it. A hole is {\em even} it it has even length, and {\em odd} otherwise. The class of even-hole-free graphs has been studied extensively (see the survey \cite{Kristina}), but many open questions remain. Among them are several algorithmic problems: \textsc{Stable Set}, \textsc{Vertex Cover}, \textsc{Dominating Set}, $k$-\textsc{Coloring} and \textsc{Coloring}. The structural complexity of this class of graphs is further evidenced by the fact that there exist
even-hole-free graphs of arbitrarily large 
tree-width \cite{ST} (even when the clique number is bounded).  Closer examination of the  constructions of \cite{ST} led the authors of \cite{ST} to make the following two conjectures (the {\em diamond} is the unique simple graph with four vertices and five edges):

\begin{conjecture}[Sintiari and Trotignon \cite{ST}]
\label{diamond}
For every integer $t$, there exists a constant $c_t$ such that every
even-hole-free graph $G$ with no diamond and no clique of size $t$ satisfies
$\tw(G)\leq c_t$. 
\end{conjecture}

\begin{conjecture}[Sintiari and Trotignon \cite{ST}]
\label{logconj}
For every integer $t$, there exists a constant $c_t$ such that every
even-hole-free graph $G$ with no clique of size $t$ satisfies
$\tw(G)\leq c_t \log |V(G)|$. 
\end{conjecture}

(In fact, \cite{ST} only states the two conjectures above for $t=4$.)
Conjecture \ref{diamond} was recently proved in \cite{TWXI}.
Here we prove a special case of
Conjecture \ref{logconj}. We remark that the full version of the conjecture
is proved, by a different set of authors, in a forthcoming paper
in the series \cite{TWXV}.
However, the contributions of the present work are of
independent interest, as we will explain later.
A {\em generalized $t$-pyramid} is a graph obtained from a certain kind of tree
(a subdivided star or a subdivided cubic caterpillar) and the line graph of a subdivided cubic caterpillar by identifying their simplicial vertices (we give a precise definition later). We prove:

\begin{theorem}
\label{logconj_pyramid}
For every  integer $t$, there exists a constant
$c_t$ such that every
even-hole-free graph $G$ with no clique of size $t$  and no generalized $t$-pyramid satisfies
$\tw(G)\leq c_t \log |V(G)|$. 
\end{theorem}
We remark that the construction of \cite{ST} shows that the logarithmic bound of Theorem \ref{logconj_pyramid} is best possible for this class. Furthermore,
before  Theorem \ref{logconj_pyramid}  was proved, 
the complexity of \textsc{Stable Set}, \textsc{Vertex Cover}, \textsc{Dominating Set}, $k$-\textsc{Coloring} and \textsc{Coloring} when restricted to this class of graphs was not known. 

Given a graph $G$, a {\em path in $G$} is an induced subgraph of $G$ that is a path. If $P$ is a path in $G$, we write $P = p_1 \dd \dots \dd p_k$ to mean that $V(P) = \{p_1, \dots, p_k\}$, and $p_i$ is adjacent to $p_j$ if and only if $|i-j| = 1$. We call the vertices $p_1$ and $p_k$ the \emph{ends of $P$}, and say that $P$ is a path \emph{from $p_1$ to $p_k$}. The \emph{interior of $P$}, denoted by $P^*$, is the set $V(P) \setminus \{p_1, p_k\}$. The \emph{length} of a path $P$ is the number of edges in $P$.  We denote by $C_k$ a cycle with $k$ vertices.

Next we describe a few types of graphs that we will need (see Figures \ref{fig:forbidden_isgs} and \ref{fig:kpyramids}).
\begin{figure}[t!]
\begin{center}
\begin{tikzpicture}[scale=0.29]

\node[inner sep=1.5pt, fill=black, circle] at (0, 3)(v1){}; 
\node[inner sep=1.5pt, fill=black, circle] at (-3, 0)(v2){}; 
\node[inner sep=1.5pt, fill=black, circle] at (3, 0)(v3){}; 
\node[inner sep=1.5pt, fill=black, circle] at (0, 0)(v4){}; 
\node[inner sep=1.5pt, fill=black, circle] at (0, -6)(v5){};

\draw[black, thick] (v1) -- (v2);
\draw[black, thick] (v1) -- (v3);
\draw[black, thick] (v1) -- (v4);
\draw[black, dashed, thick] (v2) -- (v5);
\draw[black, dashed, thick] (v3) -- (v5);
\draw[black, dashed, thick] (v4) -- (v5);

\end{tikzpicture}
\hspace{0.7cm}
\begin{tikzpicture}[scale=0.29]

\node[inner sep=1.5pt, fill=black, circle] at (-3, 3)(v1){}; 
\node[inner sep=1.5pt, fill=black, circle] at (0, 0)(v2){}; 
\node[inner sep=1.5pt, fill=black, circle] at (3, 3)(v3){}; 
\node[inner sep=1.5pt, fill=black, circle] at (-3, -6)(v4){}; 
\node[inner sep=1.5pt, fill=black, circle] at (0, -3)(v5){}; 
\node[inner sep=1.5pt, fill=black, circle] at (3, -6)(v6){};

\draw[black, thick] (v1) -- (v2);
\draw[black, thick] (v1) -- (v3);
\draw[black, thick] (v2) -- (v3);
\draw[black, dashed, thick] (v1) -- (v4);
\draw[black, dashed, thick] (v2) -- (v5);
\draw[black, dashed, thick] (v3) -- (v6);
\draw[black, thick] (v4) -- (v5);
\draw[black, thick] (v4) -- (v6);
\draw[black, thick] (v5) -- (v6);
\end{tikzpicture}
\hspace{0.7cm}
\begin{tikzpicture}[scale=0.29]

\node[inner sep=1.5pt, fill=black, circle] at (0, 3)(v1){}; 
\node[inner sep=1.5pt, fill=black, circle] at (-3, 0)(v2){}; 
\node[inner sep=1.5pt, fill=black, circle] at (3, 0)(v3){}; 
\node[inner sep=1.5pt, fill=black, circle] at (-3, -6)(v4){}; 
\node[inner sep=1.5pt, fill=black, circle] at (0, -3)(v5){}; 
\node[inner sep=1.5pt, fill=black, circle] at (3, -6)(v6){};

\draw[black, thick] (v1) -- (v2);
\draw[black, thick] (v1) -- (v3);
\draw[black, dashed, thick] (v1) -- (v5);
\draw[black, dashed, thick] (v2) -- (v4);
\draw[black, dashed, thick] (v3) -- (v6);
\draw[black, thick] (v4) -- (v5);
\draw[black, thick] (v4) -- (v6);
\draw[black, thick] (v5) -- (v6);

\end{tikzpicture}
\hspace{0.7cm}
\begin{tikzpicture}[scale=0.29]

\node[inner sep=1.5pt, fill=black, circle] at (-3, 2)(v2){}; 
\node[inner sep=1.5pt, fill=black, circle] at (3, 2)(v3){}; 
\node[inner sep=1.5pt, fill=black, circle] at (-3, -4)(v4){}; 
\node[inner sep=1.5pt, fill=black, circle] at (0, -1)(v5){}; 
\node[inner sep=1.5pt, fill=black, circle] at (3, -4)(v6){};

\draw[black, thick] (v5) -- (v2);
\draw[black, thick] (v5) -- (v3);
\draw[black, dashed, thick] (v2) -- (v4);
\draw[black, dashed, thick] (v3) -- (v6);
\draw[black, dashed, thick] (v2) -- (v3);
\draw[black, thick] (v4) -- (v5);
\draw[black, dashed, thick] (v4) -- (v6);
\draw[black, thick] (v5) -- (v6);

\end{tikzpicture}
\end{center}
%ace{-0.75cm}
\caption{Theta, prism, pyramid and an even wheel. Dashed lines represent paths of length at least one.}
\label{fig:forbidden_isgs}
\end{figure}
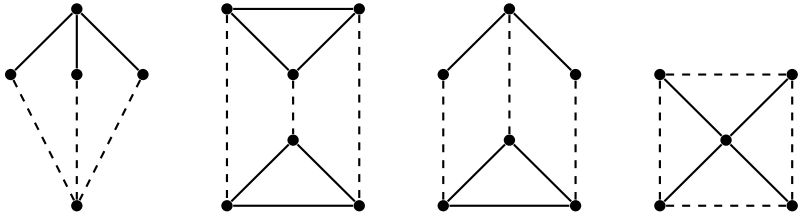
A \emph{theta} is a graph consisting of three internally vertex-disjoint paths $P_1=a \dd \hdots \dd b$, $P_2 = a \dd \dots \dd b$, and
$P_3 = a \dd \dots \dd b$, each of length at least 2, such that $P_1^*, P_2^*, P_3^*$ are pairwise anticomplete. In this case we call $a$ and $b$ the
{\em ends} of the theta.

A \emph{prism} is a graph consisting of three vertex-disjoint paths
$P_1 = a_1 \dd \dots \dd b_1$, $P_2 = a_2 \dd \dots \dd b_2$, and $P_3 = a_3 \dd \dots \dd b_3$, each of
length at least 1, such that $a_1a_2a_3$ and $b_1b_2b_3$ are triangles, and no edges exist between the paths except those of the two triangles.

Given an integer $k$, a  {\em generalized $k$-pyramid} is a graph whose vertex set is the disjoint union of $k+2$ paths $P,Q, R_1, \dots, R_k$, such that the following hold (here $P$ is the bottom path in the graphs in Figure~\ref{fig:kpyramids} and
$Q$ is the top path): 
\begin{itemize}
  \item $P \cup Q$ is a hole.
  \item For every $i \in \{1, \dots, k\}$, the path $R_i$ has ends $a_i$ and $b_i$. 
  \item For every $i \in \{1, \dots, k\}$, $a_i$ has exactly two neighbors  $x_i,y_i$ in $P$. Moreover, 
    $x_i, y_i \in P^*$ and $x_i$ is adjacent to $y_i$.
  \item For every $i \in \{1, \dots, k\}$, $b_i$ has exactly one neighbor $z_i$ in $Q$.
  \item For every $i \in \{1, \dots, k\}$, $R_i \setminus \{a_i,b_i\}$ is anticomplete to $P \cup Q$.
  \item For every $1 \leq i <j \leq k$, $R_i$ is anticomplete to $R_j$.
  \item $P$ traverses $x_1,y_1,x_2,y_2, \dots, x_k,y_k$ in this order.
  \item $Q$ traverses $z_1, z_2, \dots, z_k$ in this order (but note that some of these vertices may coincide).
 \item For every $i \in \{1, \dots, k-1\}$, we have that $y_i \neq x_{i+1}$ (and so $x_1,y_1,x_2,y_2, \dots, x_k,y_k$ are all distinct). 
\end{itemize}
A generalized $k$-pyramid where $|Q|=1$ is sometimes
referred to as a \emph{$k$-pyramid}, and a $1$-pyramid is usually called
a {\em pyramid}.

A \emph{wheel} $(H, x)$ consists of a hole $H$ and a vertex $x$ such that $x$ has at least three neighbors in $H$. A wheel $(H,x)$ is {\em even} if $x$ has an even number of neighbors on $H$.

Our main result is a slight strengthening of Theorem \ref{logconj_pyramid}. Let $\mathcal{C}$ be the
class of ($C_4$, theta, prism, even wheel)-free graphs (these are sometimes called ``$C_4$-free odd-signable graphs''). For every integer $t \geq 1$, let
$\mathcal{C}_t$ be the class of all graphs in $\mathcal{C}$ with no clique of size $t$, and let $\mathcal{C}_{tt}$ be the class of all graphs in
$\mathcal{C}_t$ that are also generalized $t$-pyramid-free.
It is easy to see that every even-hole-free graph is in $\mathcal{C}$.
We prove:

\begin{theorem}
\label{main}
For every integer $t$, there exists a constant $c_t$ such that every
 $G \in \mathcal{C}_{tt}$ satisfies
$\tw(G)\leq c_t \log |V(G)|$. 
\end{theorem}

The general idea of the proof of \ref{main} is similar to the proof of
the main result of \cite{TWIII}. However, there were several technical steps in
\cite{TWIII}, dealing with so called ``balanced vertices'', that we managed to circumvent here, using a much more elegant approach of $k$-lean tree decompositions. In addition to making the proof less technical, using this tool also allowed us to only exclude  generalized $t$-pyramids instead of  pyramids (note that every generalized $t$-pyramid contains a pyramid).

\subsection {Contributions beyond the main result}

In this paper  we introduce  new techniques that are nicer than the
techniques used earlier in the series. These include  the use of $k$-lean tree
decompositions, and the construction of Theorem~\ref{connectors} that is a
way to emulate the well-known concept of ``torsos''  in the induced subgraph
world.

Several results of the present paper are of independent interest, beyond their applications to Theorem~\ref{main}. First and foremost, in
Section~\ref{sec:connectifiers} we prove several results describing the structure of minimal connected subgraphs containing the neighbors of a (large subset
of) a given set of  vertices. These structural results are used repeatedly in
forthcoming papers in the series. Secondly, the full proof of \cite{TWXV}
uses results of  Section~\ref{cutsets} of the current paper, as well as methods developed in Sections~\ref{sec:centralbag} and \ref{sec:proof}. Finally,
this paper pushes the results of Section~\ref{sec:non-hubs} to their limit.
The bound
obtained here does not seem to hold in more general settings, and
methods of completely different nature were needed to make further progress.

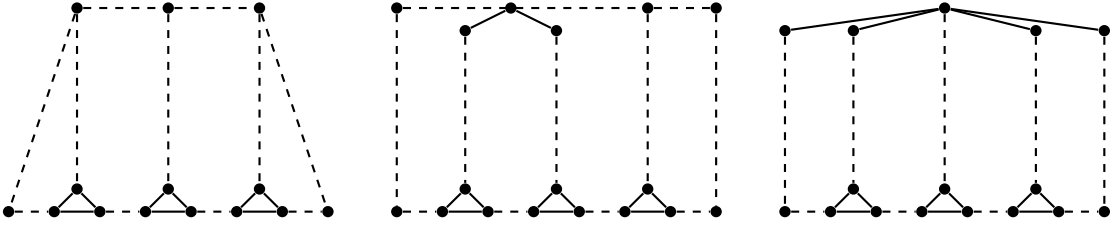
\begin{figure}[t!]
\label{fig:kpyramid}
\begin{center}
\begin{tikzpicture}[scale=0.30]

\node[inner sep=1.5pt, fill=black, circle] at (-1, 6)(v2){}; 
\node[inner sep=1.5pt, fill=black, circle] at (3,6)(v3){}; 
\node[inner sep=1.5pt, fill=black, circle] at (7,6)(v4){};
%\node[inner sep=1.5pt, fill=black, circle] at (10,6)(v16){};
\node[inner sep=1.5pt, fill=black, circle] at (-4,-3)(v5){};
\node[inner sep=1.5pt, fill=black, circle] at (-2,-3)(v6){};
\node[inner sep=1.5pt, fill=black, circle] at (-1,-2)(v7){};
\node[inner sep=1.5pt, fill=black, circle] at (0,-3) (v8){};
\node[inner sep=1.5pt, fill=black, circle] at (2,-3)(v9){};
\node[inner sep=1.5pt, fill=black, circle] at (3,-2)(v10){};
\node[inner sep=1.5pt, fill=black, circle] at (4,-3)(v11){};
\node[inner sep=1.5pt, fill=black, circle] at (6,-3)(v12){};
\node[inner sep=1.5pt, fill=black, circle] at (7,-2)(v13){};
\node[inner sep=1.5pt, fill=black, circle] at (8,-3)(v14){};
\node[inner sep=1.5pt, fill=black, circle] at (10,-3)(v15){};

\draw[black, dashed, thick] (v2) -- (v3);
\draw[black, %dashed,
  thick] (v2) -- (v5);
\draw[black,  dashed, thick] (v3) -- (v4);
\draw[black,  %dashed,
  thick] (v4) -- (v15);
\draw[black, dashed, thick] (v2) -- (v7);
\draw[black, dashed, thick] (v3) -- (v10);
\draw[black, dashed, thick] (v4) -- (v13);

\draw[black, dashed, thick] (v5) -- (v6);
\draw[black,  thick] (v6) -- (v8);
\draw[black, dashed,  thick] (v8) -- (v9);
\draw[black,  thick] (v9) -- (v11);
\draw[black, dashed, thick] (v11) -- (v12);
\draw[black, thick] (v12) -- (v14);
\draw[black, dashed, thick] (v14) -- (v15);
\draw[black, thick] (v6) -- (v7);
\draw[black,  thick] (v7) -- (v8);
\draw[black, thick] (v9) -- (v10);
\draw[black,  thick] (v10) -- (v11);
\draw[black, thick] (v12) -- (v13);
\draw[black,  thick] (v13) -- (v14);

\end{tikzpicture}
\hspace{0.5cm}
\begin{tikzpicture}[scale=0.30]

\node[inner sep=1.5pt, fill=black, circle] at (-4, 6)(v1){}; 
\node[inner sep=1.5pt, fill=black, circle] at (-1, 5)(v2){};
\node[inner sep=1.5pt, fill=black, circle] at (1, 6)(v17){};
\node[inner sep=1.5pt, fill=black, circle] at (3,5)(v3){}; 
\node[inner sep=1.5pt, fill=black, circle] at (7,6)(v4){};
\node[inner sep=1.5pt, fill=black, circle] at (10,6)(v16){};
\node[inner sep=1.5pt, fill=black, circle] at (-4,-3)(v5){};
\node[inner sep=1.5pt, fill=black, circle] at (-2,-3)(v6){};
\node[inner sep=1.5pt, fill=black, circle] at (-1,-2)(v7){};
\node[inner sep=1.5pt, fill=black, circle] at (0,-3) (v8){};
\node[inner sep=1.5pt, fill=black, circle] at (2,-3)(v9){};
\node[inner sep=1.5pt, fill=black, circle] at (3,-2)(v10){};
\node[inner sep=1.5pt, fill=black, circle] at (4,-3)(v11){};
\node[inner sep=1.5pt, fill=black, circle] at (6,-3)(v12){};
\node[inner sep=1.5pt, fill=black, circle] at (7,-2)(v13){};
\node[inner sep=1.5pt, fill=black, circle] at (8,-3)(v14){};
\node[inner sep=1.5pt, fill=black, circle] at (10,-3)(v15){};

\draw[black,  dashed, thick] (v1) -- (v4);
\draw[black,   thick] (v1) -- (v17);
\draw[black,   thick] (v4) -- (v16);
\draw[black,  thick] (v17) -- (v2);
\draw[black, thick] (v17) -- (v3);
\draw[black,  dashed, thick] (v1) -- (v16);
\draw[black, %dashed,
  thick] (v1) -- (v5);
\draw[black, dashed, thick] (v2) -- (v7);
\draw[black, dashed, thick] (v3) -- (v10);
\draw[black, dashed, thick] (v4) -- (v13);
\draw[black, %dashed,
  thick] (v16) -- (v15);
\draw[black,   thick] (v1) -- (v17);
\draw[black, dashed, thick] (v5) -- (v6);
\draw[black,  thick] (v6) -- (v8);
\draw[black, dashed, thick] (v8) -- (v9);
\draw[black,  thick] (v9) -- (v11);
\draw[black, dashed, thick] (v11) -- (v12);
\draw[black, thick] (v12) -- (v14);
\draw[black, dashed, thick] (v14) -- (v15);
\draw[black, thick] (v6) -- (v7);
\draw[black,  thick] (v7) -- (v8);
\draw[black, thick] (v9) -- (v10);
\draw[black,  thick] (v10) -- (v11);
\draw[black, thick] (v12) -- (v13);
\draw[black,  thick] (v13) -- (v14);
\end{tikzpicture}
\hspace {0.5 cm}
\begin{tikzpicture}[scale=0.30]

\node[inner sep=1.5pt, fill=black, circle] at (-4, 5)(v1){}; 
\node[inner sep=1.5pt, fill=black, circle] at (-1, 5)(v2){};
\node[inner sep=1.5pt, fill=black, circle] at (3, 6)(v17){};
%\node[inner sep=1.5pt, fill=black, circle] at (3,5)(v3){}; 
\node[inner sep=1.5pt, fill=black, circle] at (7,5)(v4){};
\node[inner sep=1.5pt, fill=black, circle] at (10,5)(v16){};
\node[inner sep=1.5pt, fill=black, circle] at (-4,-3)(v5){};
\node[inner sep=1.5pt, fill=black, circle] at (-2,-3)(v6){};
\node[inner sep=1.5pt, fill=black, circle] at (-1,-2)(v7){};
\node[inner sep=1.5pt, fill=black, circle] at (0,-3) (v8){};
\node[inner sep=1.5pt, fill=black, circle] at (2,-3)(v9){};
\node[inner sep=1.5pt, fill=black, circle] at (3,-2)(v10){};
\node[inner sep=1.5pt, fill=black, circle] at (4,-3)(v11){};
\node[inner sep=1.5pt, fill=black, circle] at (6,-3)(v12){};
\node[inner sep=1.5pt, fill=black, circle] at (7,-2)(v13){};
\node[inner sep=1.5pt, fill=black, circle] at (8,-3)(v14){};
\node[inner sep=1.5pt, fill=black, circle] at (10,-3)(v15){};

%\draw[black,  dashed, thick] (v1) -- (v16);
\draw[black,  thick] (v17) -- (v2);
\draw[black, thick] (v17) -- (v1);
\draw[black, thick] (v17) -- (v4);
\draw[black, thick] (v17) -- (v16);
\draw[black, %dashed,
  thick] (v1) -- (v5);
\draw[black, dashed, thick] (v2) -- (v7);
\draw[black, dashed, thick] (v17) -- (v10);
\draw[black, dashed, thick] (v4) -- (v13);
\draw[black, %dashed,
  thick] (v16) -- (v15);

\draw[black, dashed, thick] (v5) -- (v6);
\draw[black,  thick] (v6) -- (v8);
\draw[black, dashed, thick] (v8) -- (v9);
\draw[black,  thick] (v9) -- (v11);
\draw[black, dashed, thick] (v11) -- (v12);
\draw[black, thick] (v12) -- (v14);
\draw[black, dashed, thick] (v14) -- (v15);
\draw[black, thick] (v6) -- (v7);
\draw[black,  thick] (v7) -- (v8);
\draw[black, thick] (v9) -- (v10);
\draw[black,  thick] (v10) -- (v11);
\draw[black, thick] (v12) -- (v13);
\draw[black,  thick] (v13) -- (v14);

\end{tikzpicture}

\end{center}
%\vspace{-0.75cm}
\caption{Examples of generalized $3$-pyramids.  Dashed lines represent paths of length at least one.}
\label{fig:kpyramids}
\end{figure}

\subsection{Proof outline and organization}

Let us now discuss the main ideas of the proof of
Theorem \ref{main}. We will give precise definitions of the concepts used below later in the paper; our goal here is to sketch  a road map of where
we are going. Obtaining a tree decomposition is usually closely related to
producing  a collection of ``non-crossing separations,'' which  roughly
means that the separations
``cooperate'' with each other, and the pieces that are obtained when the graph
is simultaneously decomposed by all the separations in the collection
``line up'' to form a tree structure.

We remark that most of our arguments work for graphs in the larger class  $\mathcal{C}_t$;
we will specify the point when we need to make the assumption that
generalized $t$-pyramids are excluded.
In the case of graphs in $\mathcal{C}_t$ (as well as some other graph classes addressed in other papers of this series), there is a natural family of separations to turn to;  they correspond to special vertices of the graph called
``hubs,'' and are discussed in Section~\ref{cutsets}.
Unfortunately,
these natural separations are very far from being non-crossing, and therefore
we cannot use them in traditional ways to get tree decompositions. Similarly to
\cite{TWIII}, we use degeneracy to partition the set
of all hubs (which yields a partition of all the natural separations)
of an $n$-vertex  graph $G$ in $\mathcal{C}_t$  into collections $S_1, \dots, S_p$, where each $S_i$ is ``non-crossing''
(this property is captured in Lemma \ref{looselylaminar}), $p \leq C(t) \log n$ (where $C(t)$ only depends on $t$ and works for all $G \in \mathcal{C}_t$)
and each vertex of  $S_i$ has at most $d$ (where $d$ depends on  $t$)
neighbors in
$\bigcup_{j=i}^p S_j$. Our main result is that the treewidth of $G$  is
bounded by a  linear  function of $p+\log n$.

First we will  show that graphs in $\mathcal{C}_t$ that do not have hubs have treewidth that is bounded as a function of $t$; thus we may assume that $p>0$.
In fact, we will prove that
there exists a constant $k$, depending on $t$, such that if two 
 non-adjacent
vertices $u,v$ of a graph in $\mathcal{C}_t$ are joined by $k$ vertex-disjoint paths $P_1, \dots, P_k$, then for at least one $i \in \{1, \dots, k\}$, the neighbor of
$v$ in $P_i$ is a hub. 

We now proceed as follows. Let $m=2d+k$. We first consider a
so-called $m$-lean tree decomposition $(T, \chi)$ of $G$ (discussed in Section \ref{linkedandlean}). By standard arguments on tree decompositions,
we deduce that some bag $\chi(t_0)$ of $(T, \chi)$ is in some sense central
to the tree decomposition, and we focus on it. We can then show that all but at
most one vertex of $S_1$ has bounded degree in the torso of $\chi(t_0)$.
We would like to use this fact in order to construct a tree decomposition of the torso of $\chi(t_0)$ and then use it to obtain a tree decomposition of $G$.
Unfortunately, the torso of $\chi(t_0)$ is not a graph in $\mathcal{C}_t$.
Instead, we find an induced subgraph of $G$, which we call $\beta$, that
consists of $\chi(t_0)$ together with a collection of disjoint vertex sets
$Conn(t)$ for $t \in N_T(t_0)$, where each $Conn(t)$ ``remembers'' 
 the component of $G \setminus \chi(t_0)$ that
 meets $\chi(t)$.  Moreover, no vertex of
 $\beta \setminus \chi(t_0)$ is a hub of $\beta$, 
and all but one vertex of $S_1$ have bounded degree in $\beta$.

Next, we  decompose
 $\beta$, simultaneously, by all the separations corresponding to the hubs in $S_1$ whose degree in $\beta$ is bounded, and delete the unique vertex of
 $S_1$ of high degree (if one exists). We denote the resulting graph  by $\beta^A(S_1)$ and call it 
    the ``central bag'' for $S_1$. The parameter $p$
  is smaller for $\beta^A(S_1)$ than it is for $G$, and so we can use induction
  to obtain a bound on the treewidth of $\beta^A(S_1)$.
  We then  start with a special optimal tree decomposition of $\beta^A(S_1)$,
  where each bag is a ``potential maximal clique'' (see Section~\ref{linkedandlean}).
  Also inductively (this time on the number of vertices) we have tree
  decompositions for each component of $\beta \setminus \beta^A(S_1)$.

  Now  we use the special nature of our ``natural separations'' and properties of potential maximal cliques to combine the
  tree decompositions above into a tree decomposition of $\beta$, where the size
  of the bag only grows by an additive constant.

  Then we repeat a similar procedure to combine the tree decomposition of
  $\beta$ that we just obtained with tree decompositions of the components of $G \setminus \chi(t_0)$ to obtain a tree decomposition of $G$, where again
  the size of the bag only grows by an additive constant.

  Let us now discuss how we obtain the bound on the growth of a bag.
  In the first step of the growing process, we   
  add to each existing bag $B$  the neighbor sets of the vertices of $S_1 \cap B$.     The number of vertices of $S_1$  in  each bag is bounded
    by Theorem \ref{boundhubs}, because no vertex of $S_1$ is a hub in
    $\beta^A(S_1)$ (this is proved in Theorem \ref{A_centralbag}).

    In the second growing step, we first turn the tree decomposition of
    $\beta$ that we just constructed into a tree decomposition of the same
    width with the additional property that every bag is a potential maximal clique of $G$.
 Next,   for each  bag $B$ of this tree decomposition, 
    and for every $t \in N_T(t_0)$ such that $B \cap Conn(t) \neq \emptyset$, we add to $B$
    the adhesion $\chi(t_0) \cap \chi(t)$. One of the properties
    of $m$-lean tree decompositions is that the size of each adhesion is bounded.
    The number of adhesions added to a given bag is again bounded by
    Theorem \ref{boundhubs} since distinct sets $Conn(t)$ are pairwise disjoint and anticomplete to each other, and no vertex of $Conn(t)$ is a hub.

    Theorem \ref{boundhubs} is the only result in the paper that uses the
    stronger assumption that  generalized $t$-pyramids are excluded.

    The paper is organized as follows. In Section~\ref{linkedandlean},
    we discuss several types of tree decompositions that we use in this paper.
        In Section~\ref{cutsets}, we summarize 
  results guaranteeing the existence of useful separations.
  In Section~\ref{sec:centralbag}, we discuss the construction of
  the graphs $\beta$ and $\beta^A(S_1)$, and how to use their tree
  decompositions to obtain a tree decomposition of $G$.
  In Section~\ref{sec:connectifiers}, we analyze the structure of minimal separators in graphs of $\mathcal{C}_t$.
  In Section~\ref{boundhubs}, we use the results of Section~\ref{sec:connectifiers}
  to obtain a bound on the size of a stable set of non-hubs in a potential maximal clique of a graph in $\mathcal{C}_{tt}$. 
  Section~\ref{sec:proof} puts together the results of all the previous sections
  to prove Theorem \ref{main}.
  Finally, Section~\ref{algsec}
  discusses algorithmic consequences of Theorem \ref{main}.
 
\section{Special tree decompositions and connectivity} \label{linkedandlean}

In this section we discuss several known results related to connectivity, and
a describe a few special kinds of tree decompositions.
Let $G$ be a graph. Let $X, Y, Z$ be subsets of $V(G)$. By a \emph{path from $Y$ to $Z$} we mean a path from some $y \in Y$ to some $z \in Z$. We say that
{\em $X$ separates $Y$ from $Z$} (in $G$) if every path $P$ with an end in $Y$ and an end in $Z$ satisfies $P \cap X \neq \emptyset$. 
%if there exist distinct components $C_Y,C_Z$ of $G\setminus X$ such that $Y \setminus X \subseteq C_Y$ and $Z \setminus X \subseteq C_Z$.
In this case we also say that {\em $Y$ is separated from $Z$ by $X$}.
We start by recalling a classical
result of Menger (please note that even though the next two theorems are usually stated in terms of paths that are not induced, the formulation below is equivalent by taking paths with minimal vertex sets subject to preserving the ends):
\begin{theorem}[Menger \cite{Menger}]\label{Menger}
  Let $k\geq 1$ be an integer, let $G$ be a graph and let
  $X, Y  \subseteq V(G)$ with $|X|=|Y|=k$. Then either there exists 
  $M\subseteq V(G)$  with $|M|<k$ such that $M$ separates $X$ from $Y$, or
   there are $k$ pairwise vertex-disjoint paths in $G$ from $X$ to $Y$.
\end{theorem}

Theorem \ref{Menger} immediately implies:

\begin{theorem}[Menger \cite{Menger}]\label{Menger_vertex}
   Let $k\geq 1$ be an integer, let $G$ be a graph and let $u,v \in G$ be distinct and non-adjacent. Then either there exists a set $M\subseteq G\setminus \{u,v\}$ with $|M|<k$ such that $M$ separates $u$ and $v$ in $G$, or there are $k$ pairwise internally vertex-disjoint paths in $G$ from $u$ to $v$.
\end{theorem}

For a tree $T$ and vertices $t,t' \in V(T)$,
we denote by $tTt'$ the unique path of $T$ from $t$ to $t'$.
Let $(T, \chi)$ be a tree decomposition of a graph $G$.
For every $x\in V(T)$,  the \textit{torso at $x$}, denoted by $\hat{\chi}(x)$, is the graph obtained from the bag $\chi(x)$ by, for each $y\in N_T(x)$, adding an edge between every two non-adjacent vertices $u,v\in \chi(x) \cap \chi(y)$.
For every $uv \in E(T)$, the {\em adhesion at $uv$}, denoted by $adh(uv)$, is the set $\chi(u) \cap \chi (v)$. We define $adh(T,\chi)=0$ if $E(T)=\emptyset$, and  $adh(T, \chi)=\max_{uv \in E(T)} |adh(uv)|$ in all other cases.

In the proof of Theorem \ref{main}, we will use several special kinds of tree decompositions that we explain now.

\subsection{Potential Maximal Cliques}

For a graph $G$ and a set $F \subseteq \binom{V(G)}{2} \setminus E(G)$, we denote by $G+F$ the graph obtained from $G$ by making the pairs in $F$ adjacent, that is, $E(G+F) = E(G) \cup F$. 
A set $F \subseteq \binom{V(G)}{2} \setminus E(G)$ 
is a \emph{chordal fill-in} of $G$ if $G+F$ is chordal; in this case, $G+F$ is a {\em chordal completion} of $G$. A chordal fill-in (and the corresponding chordal completion) is \emph{minimal} if it is inclusion-wise minimal. 

Let $X \subseteq V(G)$. The set $X$ is a \emph{minimal separator} if there exist $u, v \in V(G)$ such that $u$ and $v$ are in different connected components of $G \setminus X$, and $u$ and $v$ are in the same connected component of $G \setminus Y$ for every $Y \subsetneq X$.
A component $D$ of $G \setminus X$  is a \emph{full component} for $X$ if $N(D) = X$. It is well-known that a set $X \subseteq V(G)$ is a minimal separator if and only if there are at least two distinct full components for $X$.

A \emph{potential maximal clique} (PMC) of a graph $G$ is a set $\Omega \subseteq V(G)$ such that $\Omega$ is a maximal clique of some minimal chordal completion $G+F$ of $G$. The following result characterizes PMCs: 

\begin{theorem}[Bouchitt\'e and Todinca \cite{BouchitteT01}]
\label{thm:PMC_characterization}
A set $\Omega \subseteq V(G)$ is a PMC of $G$ if and only if: 
\begin{enumerate}
    \item \label{PMC1} for all distinct $x, y \in \Omega$ with $xy \not \in E(G)$, there exists a component $D$ of  $G \setminus \Omega$ such that $x, y \in N(D)$; and 
    \item for every component $D$ of $G \setminus \Omega$ it holds that $N(D) \neq \Omega$ (that is, there are no full components for $\Omega$).
\end{enumerate}
\end{theorem}
%If $\Omega \subseteq V(G)$ and $D$ is a component of $G \setminus \Omega$ with $x, y \in N(\Omega)$ non-adjacent (as in \eqref{PMC1} above), we say that $D$ \emph{covers} the non-edge $xy$. 

We also need the following result relating PMCs and
minimal separators: 

\begin{theorem}[Bouchitt\'e and Todinca \cite{BouchitteT01}]
\label{prop:PMC_adhesions_are_seps}
Let $\Omega \subseteq V(G)$ be a PMC of $G$. Then, for every component $D$ of  $G \setminus \Omega$, the set $N(D)$ is a minimal separator of $G$. 
\end{theorem}

We remind the reader of the following well-known property of tree decompositions:

\begin{theorem}[see Diestel \cite{diestel}]
  \label{cliqueinbag}
  Let $G$ be a graph, let $K$ be a clique of $G$, and let $(T, \chi)$ be a tree decomposition of $G$. Then, there is $v \in V(T)$ such that $K \subseteq \chi(v)$.
  \end{theorem}
Let us say that a tree decomposition $(T, \chi)$ of a graph $G$ is {\em structured} if $\chi(v)$ is a PMC of $G$ for every $v \in V(T)$.
We need the following well-known result which allows us to turn a tree decomposition into a structured tree decomposition; we  include the proof for completeness.

\begin{theorem}
  \label{structured}
  Let $G$ be a graph and let $(T, \chi)$ be a tree decomposition of $G$.
  There exists a structured tree decomposition $(T',\chi')$ of $G$
  such that for every $v' \in T'$ there exists $v \in T$ with
  $\chi'(v') \subseteq \chi(v)$.
  \end{theorem}

\begin{proof}
  Let $(T,\chi)$ be a tree decomposition of $G$.
  It is easy to check that
  the graph $G'$ obtained from $G$ by adding all edges $xy$ such
  that $x,y \in \chi(v)$ for some $v \in V(T)$ is chordal and that
  $(T, \chi)$ is a tree decomposition of $G'$. It follows that
  there exists a minimal chordal fill-in $F$ of $G$ such that
  $F \subseteq E(G') \setminus E(G)$; let $G''=G+F$.
  In particular, every clique of $G''$ is a subset of a clique of $G'$.
  Since by Theorem \ref{cliqueinbag} every clique of $G'$ is contained in a bag $\chi(v)$ for some $v \in V(T)$, it follows that every clique of $G''$ is contained in a bag $\chi(v)$ for some $v \in T$.

  Next, since $G''$ is chordal, there is a tree decomposition $(T'',\chi'')$ of $G''$  such that $\chi''(v)$ is a clique of $G''$ for every $v \in V(T'')$.
  By choosing $(T'',\chi'')$ with $|V(T'')|$ minimum, we may assume that
  there do not exist $v_1,v_2 \in T''$ such that $\chi''(v_1) \subseteq \chi''(v_2)$. This implies that $\chi''(v)$ is a maximal clique of $G''$ (and therefore    a PMC of $G$)  for every $v \in V(T'')$. 
  Lastly, since $G$ is a subgraph of $G''$, it follows that $(T'',\chi'')$   is a tree decomposition of $G$.     This proves Theorem \ref{structured}.
\end{proof}

\subsection{Lean tree decompositions}

Let $k>0$ be an integer. A tree decomposition $(T, \chi)$ is called {\em $k$-lean} if the following hold: 
\begin{itemize}
    \item $adh(T,\chi) < k$; and 
    \item for all $t,t' \in V(T)$ and sets $Z \subseteq \chi(t)$ 
and $Z' \subseteq \chi(t')$ with $|Z|=|Z'| \leq k$, either $G$ contains $|Z|$ disjoint paths from $Z$ to $Z'$, or some edge $ss'$ of $tTt'$ satisfies $|adh(ss')| <|Z|$.
\end{itemize}

For a tree $T$ and an edge $tt'$ of $T$, we  denote by
$T_{t \rightarrow t'}$ the
 component of $T \setminus t$ containing $t'$. Let
 $G_{t \rightarrow t'}=G[\bigcup_{v \in T_{t \rightarrow t'}} \chi(t)]$.
    A tree decomposition
$(T,\chi)$ is {\em tight} if for every edge $tt' \in E(T)$  there is a component
$D$ of $G_{t \rightarrow t'} \setminus \chi(t)$ such that
$\chi(t) \cap \chi(t') \subseteq N(D)$ (and therefore $\chi(t) \cap \chi(t') = N(D)$). 

    The following definition first appeared in  \cite{BD}. Given a tree decomposition $(T, \chi)$ of an $n$-vertex graph $G$, its \emph{fatness} is the vector $(a_n, \dots, a_0)$ where $a_i$ denotes the number of bags of $T$ of size $i$. A tree decomposition $(T, \chi)$ of $G$ is \emph{$k$-atomic} if $adh(T, \chi) < k$ and the fatness of $(T, \chi)$ is lexicographically minimum among all tree decompositions of $G$ with adhesion less than $k$.

The following is immediate from the definition:
\begin{theorem} 
  \label{atomicexists}
  For every $k>1$, every graph admits a $k$-atomic tree decomposition.
\end{theorem}
 
It was observed in \cite{kblockpaper} that 
\cite{BD} contains a proof of the following:

\begin{theorem} [Bellenbaum and Diestel \cite{BD}, see 
 Carmesin, Diestel, Hamann, Hundertmark \cite{kblockpaper}, see also Wei{\ss}auer \cite{Weissauer}]
  \label{atomictolean}
  Every $k$-atomic tree decomposition is $k$-lean.
\end{theorem}

The same proof also gives the following, which is in fact part of the definition of $k$-leanness in \cite{kblockpaper}: 
\begin{theorem}[Bellenbaum and Diestel \cite{BD}, see 
 Carmesin, Diestel, Hamann, Hundertmark \cite{kblockpaper}, see also Wei{\ss}auer \cite{Weissauer}]
  \label{atomictooptimal}
  Let $(T,\chi)$ be a $k$-atomic tree decomposition of $G$ and let
  $tt'$ be an edge of $T$. Then  both $\chi(t) \setminus \chi(t') \neq \emptyset$ and
    $\chi(t') \setminus \chi(t) \neq \emptyset$. 
\end{theorem}

We also need the following: 
\begin{theorem} [Wei{\ss}auer \cite{Weissauer}]
  \label{atomictotight}
  Every $k$-atomic tree decomposition is tight. 
\end{theorem}

We also  have  the following, which can be easily deduced  by combining Lemmas~7 and 9 of Wei\ss auer
\cite{Weissauer} and using Theorem \ref{atomictolean}. 

\begin{theorem}
\label{bigdegnonsep1}
  Let $G$ be a graph, let $k \geq 3$ and let $(T, \chi)$ be
  $k$-atomic tree decomposition of $G$. Let $t \in V(T)$. If
  $u,v \in \chi(t)$ have degree at least $(2k-2)(k-2)$ in $\hat\chi(t)$, then
  $u$ and $v$ are not separated in $G$ by a set $X \subseteq V(G) \setminus \{u, v\}$ of size less than $k$.
\end{theorem}

Using Theorem \ref{Menger_vertex}, we deduce:

\begin{theorem}
\label{bigdegnonsep}
  Let $G$ be a graph, let $k \geq 3$ and let $(T, \chi)$ be
  $k$-atomic tree decomposition of $G$. Let $t \in V(T)$. If
  $u,v \in \chi(t)$ are non-adjacent vertices, each of  degree at least $(2k-2)(k-2)$ in $\hat\chi(t)$, then
there are $k$ pairwise internally vertex-disjoint paths in $G$ from $u$ to $v$.
\end{theorem}

We finish this subsection with  a theorem about tight tree decompositions in theta-free graphs. Note that by Theorem \ref{atomictotight}, the following result applies in particular to $k$-atomic tree decompositions for every $k$.

 A \emph{cutset} $C \subseteq V(G)$ of $G$ is a (possibly empty) set of vertices such that $G \setminus C$ is disconnected. A {\em clique cutset} is a cutset that is either empty or a clique.
\begin{theorem}
  \label{connectedbranches}
  Let $G$ be a theta-free graph and assume that $G$ does not admit a clique
  cutset.
  Let $(T, \chi)$ be a tight tree decomposition of $G$.
  Then for every  edge $t_1t_2$ of $T$ the graph
  $G_{t_1\rightarrow t_2} \setminus \chi(t_1)$ is connected and
   $N(G_{t_1\rightarrow t_2} \setminus \chi(t_1))=\chi(t_1) \cap \chi(t_2)$.
  Moreover, if $t_0,t_1,t_2 \in V(T)$ and $t_1,t_2 \in N_T(t_0)$, then
  $\chi(t_0) \cap \chi(t_1) \neq \chi(t_0) \cap \chi(t_2)$.
\end{theorem}

\begin{proof}
  Suppose that $G_{t_1 \rightarrow t_2} \setminus \chi(t_1)$ is not connected. Then there exists a component $D_1$ of   $G_{t_1 \rightarrow t_2} \setminus \chi(t_1)$ such that   $N(D_1)= \chi(t_1) \cap \chi(t_2)$.
  Let $D_0$ be a component of  $G_{t_1 \rightarrow t_2} \setminus \chi(t_1)$
  different from $D_1$.
  Since $(T,\chi)$ is tight, there exists a component $D_2$ of 
  $G_{t_2 \rightarrow t_1} \setminus \chi(t_2)$ such that $N(D_2)=\chi(t_1) \cap \chi(t_2)$.
  Since $N(D_0)$ is not a clique cutset in $G$,
  there exist non-adjacent $x,y \in N(D_0) \subseteq \chi(t_1) \cap \chi(t_2)$.
  But now we get a theta with ends $x,y$ and paths with interiors in
  $D_0,D_1$ and $D_2$, respectively, a contradiction. This proves that
  $G_{t_1 \rightarrow t_2} \setminus \chi(t_1)$ is  connected. 
  
  To prove the second assertion, let  $t_1,t_2 \in N_T(t_0)$ and assume that
  $\chi(t_0) \cap \chi(t_1) = \chi(t_0) \cap \chi(t_2)$. Then,   there exist $v_0 \in \chi(t_0) \setminus \chi(t_2)$ and $v_2 \in \chi(t_2) \setminus \chi (t_0)$. Since $\chi(t_0) \cap \chi(t_1) = \chi(t_0) \cap \chi(t_2)$ and $\chi(t_1) \cap \chi(t_2) \subseteq \chi(t_0)$, it follows that $v_0, v_2 \not \in \chi(t_1)$. 
  But then  $v_0$ and $v_2$ belong to different components of
  $G_{t_1 \rightarrow t_0} \setminus \chi(t_1)$, contrary to the claim of the previous paragraph. This proves
    Theorem \ref{connectedbranches}.
%  , \dots, D_s$ be the components of
%  $G_{t_0 \rightarrow t} \setminus \chi(t_0)$. Let $T_1, \dots, T_s$
%  be $s$ copies of  $T_{t_0 \rightarrow t}$ where the copy of
%  $v \in  T_{t_0 \rightarrow t}$ is denoted by $v^i$. 
%  Let $T'$ be obtained from $T \setminus T_{t_0 \rightarrow t}$
%  by adding $T_1, \dots, T_s$ and making $t_0$ adjacent to
%  $t^1, \dots, t^s$. Now let $\chi'(v^i)=\chi(v) \cap N[D_i]$ for
%  every $v \in T_{t_0 \rightarrow t}$, and let $\chi'(v)=\chi(v)$
%  for every $v \not \in  T_{t_0 \rightarrow t}$. It is easy to
%  see that $(T',\chi')$ is an $m$-lean and tight tree decomposition of
 % $G$.
 % Repeating this construction for every vertex of $N_T(t_0)$,
 % \ref{connectedbranches} follows.
  \end{proof}

\subsection{Centers of tree decomposition}

We discuss another important feature of tree decompositions that we need. 
Let $G$ be an $n$-vertex  graph and let $(T, \chi)$ be a tree decomposition of $G$.
A vertex $t_0$ of $T$ is a  {\em center} of $(T, \chi)$ if for every
$t' \in N_T(t_0)$ we have
$|G_{t_0 \rightarrow t'} \setminus \chi(t_0)| \leq \frac{n}{2}$. The following lemma is a analogous to the standard proof that every tree has a centroid. 

\begin{theorem}
  \label{center}
  Let $(T,\chi)$ be a tree decomposition of a graph $G$. Then $(T, \chi)$ has a center. % of $G$ with $adh(T,\chi) < m$. If $G$ does not admit a  balanced separator of size less than  $m$, then $T$ has a center.
\end{theorem}

\begin{proof}
Write $|V(G)|=n$.
  Let $D$ be the directed graph obtained from $T$ as follows. Let $tt'$ be an
  edge of $T$; direct $tt'$ from $t$ to $t'$ if
  $|G_{t \rightarrow t'} \setminus \chi(t)| > \frac{n}{2}$. As $G_{t \rightarrow t'} \setminus \chi(t)$ is disjoint from $G_{t' \rightarrow t} \setminus \chi(t')$ (because $G_{t \rightarrow t'} \cap G_{t' \rightarrow t} = \chi(t) \cap \chi(t')$), this does not prescribe conflicting orientations for edges of $T$.  For each edge whose direction is not determined by this, direct it arbitrarily. 
 
  It follows that  $D$ has a sink  $t_0$. Let $t' \in N_T(t_0)$. Then, because the edge $t_0t'$ is not directed from $t_0$ to $t'$, it follows that $|G_{t_0 \rightarrow t'} \setminus \chi(t_0)| \leq \frac{n}{2}$. 
  This proves Theorem \ref{center}.
\end{proof}

\subsection{Small separators}

We conclude this section by proving one more straightforward and well-known lemma about tree decompositions. 

\begin{lemma} \label{lem:septotd}
    Let $G$ be a graph. Let $X \subseteq V(G)$, and let $D_1, \dots, D_s$ be the components of $G \setminus X$. Then $\tw(G) \leq |X| + \max_{i \in \{1, \dots, s\}} \tw(D_i)$. 
\end{lemma}
\begin{proof}
    For every $i$, we let $(T_i, \chi_i)$ be a tree decomposition of $D_i$ of width
  $\tw(D_i)$. Let $T$ be a tree obtained from the union of $T_1, \dots, T_s$
  by adding a new vertex $t$ and making $t$ adjacent to exactly one vertex of
  each $T_i$. We now construct a tree decomposition $(T, \chi)$ of $G$.
  Let $\chi(t)=X$, and for every $t' \in T \cap T_i$ let
  $\chi(t')=\chi_i(t) \cup X$.
  It is easy to check that $(T, \chi)$ is a tree decomposition of $G$,
  and $\width (T, \chi) \leq \max_{i \in \{1, \dots, s\} }\tw(D_i) + |X|$. 
\end{proof}

\section{Star cutsets, wheels and blocks } \label{cutsets}

A {\em star cutset} in a graph $G$ is a cutset $S\subseteq V(G)$ such that  either $S=\emptyset$ or for some $x\in S$, $S\subseteq N[x]$.

Recall that a \emph{wheel} $(H, x)$ of $G$ consists of a hole $H$ and a vertex $x$ that has at least three neighbors in $H$ (and therefore $x \not \in H$). A \emph{sector} of $(H,x)$ is a path $P$ of $H$ whose ends are distinct and adjacent to $x$, and such that $x$ is anticomplete to $P^*$. A sector $P$ is a \emph{long sector} if $P^*$ is non-empty.  We now define several types of wheels that we will need. 

A wheel $(H, x)$
is a \emph{universal
wheel} if $x$ is complete to $H$. A wheel $(H, x)$ is a \emph{twin wheel} if $N(x) \cap H$ induces a path of length two.
%If $(H, x)$ is a twin wheel and $x_1 \dd x_2 \dd x_3$ is the path of length two induced by $N(x) \cap H$, we say $x_2$ is the \emph{clone of $x$ in $H$}. Note that if $(H, x)$ is a twin wheel and $x_2$ is the clone of $x$ in $H$, then $((H \setminus x_2) \cup x, x_2)$ is also a twin wheel.
A wheel $(H, x)$ is a \emph{short pyramid} if $|N(x) \cap H| = 3$ and $x$ has exactly two adjacent neighbors in $H$. A wheel is \emph{proper} if it is not a twin wheel or a short pyramid. We say that $x \in V(G)$ is  a {\em wheel center} or a {\em hub} if there exists $H$ such that 
$(H, x)$ is a proper wheel in $G$.
 We denote by $\Hub(G)$ the set of all hubs of $G$.

We need the following result, which was observed in \cite{TWI}:

\begin{theorem}[Abrishami, Chudnovsky, Vu\v{s}kovi\'c \cite{TWI}]\label{wheelstarcutset}
  Let $G \in \mathcal{C}$   and let $(H,v)$ be a proper  wheel in $G$.
  Then there is no component  $D$ of $G \setminus N[v]$
such that $H \subseteq N[D]$. 
\end{theorem}

The majority of this paper is devoted to dealing with hubs and star cutsets arising from them in graphs in $\mathcal{C}$, but in the remainder of this section we focus on the case when $\Hub(G)=\emptyset$. To do that, we combine several earlier results from this series.

Let $k$ be a positive integer and let $G$ be a graph.
%For a pair $x,y \in V(G)$
%let $G^{xy}=G$ if $xy \not \in E(G)$, and let $G^{xy}$ be the graph obtained from $G$ be deleting the edge $xy$ (where $V(G^{xy})=V(G)$) if $xy \in E(G)$.
A \textit{$k$-block} in $G$ is a set $B$ of at least $k$ vertices in $G$ such that for every non-adjacent pair
$\{x,y\} \subseteq B$, there exists a collection $\mathcal{P}_{\{x,y\}}$ of at least $k$ distinct and pairwise internally vertex-disjoint paths in $G^{xy}$ from $x$ to $y$.
A slight strengthening of the following was proved in \cite{TWVIII}:

\begin{theorem} [Abrishami, Alecu, Chudnovsky, Hajebi, Spirkl \cite{TWVIII}]
    \label{noblocksmalltw_Ct_1}
  For all integers $k,t\geq 1$, there exists an integer $\beta=\beta(k,t)$ such that if $G \in \mathcal{C}_t$ and $G$ has   no $k$-block, then
  $\tw(G) \leq \beta(k,t)$.
\end{theorem}

We also need the following result from \cite{TWVIII}:

\begin{theorem}  [Abrishami, Alecu, Chudnovsky, Hajebi, Spirkl \cite{TWVIII}]
  \label{bananastructure}
    For all integers $\nu, t\geq 1$, there exists an integer $\psi=\psi(t,\nu)\geq 1$ with the following property. Let $G\in \mathcal{C}_t$, let $a,b\in G$ be distinct and non-adjacent and let $\{P_i:i\in [\psi]\}$ be a collection of $\psi$ pairwise internally disjoint paths in $G$ from $a$ to $b$. For each $i\in [\nu]$, let $a_i$ be the neighbor of $a$ in $P_i$ (so $a_i\neq b$). Then there exists $I\subseteq [\phi]$ with $|I|=\nu$ for which the following holds.
    \begin{itemize}
    \item $\{a_i:i\in I\}\cup \{b\}$ is a stable set in $G$.
    \item For all $i,j\in I$ with $i<j$, $a_i$ has a neighbor in $P_j^*\setminus \{a_j\}$.
    \end{itemize}
\end{theorem}

Additionally, we use a result of \cite{TWIV}:

\begin{theorem}[Abrishami, Alecu, Chudnovsky, Hajebi, Spirkl \cite{TWIV}]
\label{lemma:common_nbrs}
Let $G$ be a (theta,  even wheel)-free graph, let $H$ be a hole of $G$, and let $v_1, v_2 \in V(G) \setminus V(H)$ be adjacent vertices each with at least two non-adjacent neighbors in $H$. Then, $v_1$ and $v_2$ have a common neighbor in $H$. 
\end{theorem}

We also remind the reader the following well-known version of Ramsey's theorem
\cite{setRamsey}:

\begin{theorem} \label{smallanti}
For all positive integers $a,b,c$ , there is a positive integer $M = M(a,b,c)$ such
that if $G$ is a graph with no $K_a$ and no induced subgraph isomorphic to
$K_{b,b}$, and $G$
contains a collection $\mathcal{M}$ of $M$  pairwise disjoint subsets of vertices, each of
size at most $c$,  then some two members of $\mathcal{M}$ are anticomplete
to each other.
\end{theorem}

Next we show:

\begin{theorem}
  \label{banana}
  For every  integer $t \geq 1$ there exists an integer $k=k(t)$ such that
    if $G \in \mathcal{C}_t$ and $x,y \in V(G)$ are non-adjacent and
    $N(x) \cap \Hub(G)=\emptyset$, then there do not exist
    $k$ pairwise internally vertex-disjoint paths in $G$ from $x$ to $y$.
  \end{theorem}

  \begin{proof}
    Let $M=M(t,2,2)$ be as in Theorem~\ref{smallanti}
    and let $\nu=\psi(\psi(M,t),t)$ be as in Theorem~\ref{bananastructure}.
    Let $k=\psi(6\nu+3,t)$ be as in Theorem~\ref{bananastructure}, and suppose that
    there are $k$ vertex disjoint paths from $x$ to $y$ in $G$.
    Let $I$ be as in the conclusion of Theorem~\ref{bananastructure}
    applied with $a=x$ and $b=y$; renumbering the paths if necessary we
    may assume that $I=\{1, \dots, 6\nu+3\}$.
    Then each of $a_1,a_2,a_3$ has a neighbor in each of the paths
    $P_4, \dots, P_{6\nu+3}$. 
        For $i \geq 4$, let $Q_i$ be a minimal subpath of $P_i$
        such that all of $a_1,a_2,a_3$ have neighbors in $Q_i$; let $l_i$ and
        $r_i$ be the ends of $Q_i$. By the minimality of $Q_i$, for every
        $i$ there exist distinct
        $s, r \in \{a_1, a_2,a_3\}$ such that $l_i$ is the unique neighbor
        of $s$ in $Q_i$, and $r_i$ is the unique neighbor of $r$ in $Q_i$.
        By permuting $a_1$, $a_2$  and $a_3$ if necessary, we may assume that
there exists $J \subseteq \{4, \dots, 6 \nu +3\}$ with $|J|=\nu$ such that for every $i \in J$,
        $l_i$ is the unique neighbor of $a_1$ in $Q_i$,
and $r_i$ is the unique neighbor of $a_3$ in $Q_i$. Renumbering $a_4, \dots a_{6\nu+3}$ we write $J=\{4, \dots, \nu +3\}$.
For every $i \in J$ let $H_i$ be the hole
$a_1 \dd l_i \dd Q_i \dd r_i \dd a_3 \dd x \dd a_1$. Since $(H_i,a_2)$ is not
a proper wheel (recall  that  $N(x) \cap \Hub(G)=\emptyset$), and  since $H_i \cup a_2$
is not a theta, it follows that $a_2$ has exactly two neighbors $x_i,y_i$ in
$Q_i$ and they are adjacent. We may assume that $Q_i$ traverses $l_i,x_i,y_i,r_i$ in this order.

We now apply Theorem~\ref{bananastructure} again to the
paths $a_1 \dd l_i \dd Q_i \dd x_i \dd a_2$  (with $i \in \{4, \dots, \nu+3\}$)
from $a_1$ to $a_2$ with $a=a_2$ and
$b=a_1$. Let $K$ be the
set  of indices as in the conclusion of the theorem, so $|K|=\psi(M,t)$.
Now apply  Theorem~\ref{bananastructure} for the third time, to the 
paths $a_3 \dd r_i \dd Q_i \dd y_i \dd a_2$  (with $i \in K$) from $a_2$ to $a_3$ with $a=a_2$ and $b=a_3$. Let $L$ be the set of indices as in the conclusion of the theorem. Since $G \in \mathcal{C}_t$, applying Theorem~\ref{smallanti}
to the collections $\mathcal{M}=\{\{x_i,y_i\} \; : \; i \in L\}$ we deduce
that  there exist two values $i<j \in L$
such that
\begin{itemize}
\item $\{x_i,y_i\}$ is anticomplete to $\{x_j,y_j\}$, and 
\item $x_i$ has a neighbor in
  $l_j \dd  Q_j \dd x_j$, and
\item $y_i$ has a neighbor in
  $y_j \dd Q_j \dd r_j$.
\end{itemize}
We claim that $x_i$ has exactly two neighbors in $Q_j$ and they are adjacent.
Suppose first that $x_i$ has a unique neighbor $u$ in $Q_j$.
Then $u \in l_j \dd Q_j \dd x_j$, and
the hole $l_j \dd Q_j \dd x_j \dd a_2 \dd x \dd a_1 \dd l_j$ together
with $x_i$ is a theta with ends $a_2,u$, a contradiction. Next suppose that
$x_i$ has two neighbors in $Q_j$ that are non-adjacent to each other. Then by Theorem~\ref{lemma:common_nbrs} applied to $H_j$, $a_2$ and $x_i$, it follows that
$a_2$ and $x_i$ have a common neighbor in $H_j$, and therefore $x_i$
is adjacent to one of $x_j,y_j$, a contradiction. We deduce that
$x_i$ has exactly two neighbors $p_j,q_j$ in $Q_j$ and they are adjacent. Since
$x_i$ is non-adjacent to $x_j$ and has a neighbor in $l_j \dd Q_j \dd x_j$,
it follows that $p_j, q_j \in l_j \dd Q_j \dd x_j$.
Similarly $y_i$ has exactly  two neighbors $s_j,t_j $ in $Q_j$,
$s_j$ is adjacent to $t_j$, and $s_j,t_j \in y_j \dd Q_j \dd r_j$.
We may assume that $Q_j$ traverses $l_j,p_j,q_j,x_j,y_j,s_j,t_j,r_j$ in
this order (possibly $l_j=p_j$ or $t_j=r_j$). But now we get a prism with triangles $x_ip_jq_j$ and $y_it_js_j$
and paths $x_i \dd y_i$, $q_j \dd Q_j \dd s_j$ and $p_j \dd Q_j \dd l_j \dd a_1 \dd x \dd a_3 \dd r_j \dd Q_j \dd t_j$, a contradiction. This proves Theorem~\ref{banana}.
  \end{proof}
    
  From Theorems \ref{noblocksmalltw_Ct_1} and \ref{banana} we deduce:

  \begin{theorem}
  \label{noblocksmalltw_Ct}
  For every  integer $t$, there exists an integer $\gamma=\gamma(t)$ such that every $G \in \mathcal{C}_t$ with $\Hub(G)=\emptyset$ satisfies $\tw(G) \leq \gamma$.
\end{theorem}

\section{Stable   sets  of safe hubs} \label{sec:centralbag}

Let $c \in [0, 1]$. A set $X \subseteq V(G)$ is a {\em $c$-balanced separator} if
$|D| \leq  c|V(G)|$ for every component $D$ of $G \setminus X$. The set $X$ is a {\em balanced separator} if $X$ is a $\frac{1}{2}$-balanced separator.

Let $t,d$ be integers. We make the following assumptions throughout this section.
Let $G \in \mathcal{C}_t$ with $|V(G)|=n$. 
Let $m=k+2d$ where $k = k(t)$ is as in Theorem \ref{banana},
and assume that $G$ does not have a  balanced separator 
of size less than $m$. Let $(T,\chi)$ be an $m$-atomic
tree decomposition of $G$. We say that
    a vertex $v$ is {\em $d$-safe} if $|N(v) \cap \Hub(G)| \leq d$.

By Theorems \ref{atomictolean} and \ref{atomictotight}, we have that $(T,\chi)$ is  tight and $m$-lean.
By Theorem \ref{center}, there exists $t_0 \in T$ such that $t_0$ is a center for $T$.
A vertex $v \in V(G)$ is {\em $t_0$-cooperative}
if either $v \not\in \chi(t_0)$, or  $\deg_{\hat\chi(t_0)}(v) < 2m(m-1)$. 
We show that  $t_0$-cooperative vertices have the
following important property.

    \begin{lemma}
      \label{noncoop}
      If $u, v \in \chi(t_0)$ are $d$-safe and not $t_0$-cooperative, then $u$ is adjacent to $v$.
    \end{lemma}

    \begin{proof}
      Suppose that $u$ and $v$ are non-adjacent.
      Since $u,v$ are not $t_0$-cooperative, it follows that both $u,v$ have degree at least $2m(m-1)$ in $\hat\chi(t_0)$.
      Since $(T, \chi)$ is an $m$-atomic tree decomposition, Theorem \ref{bigdegnonsep} implies that 
there are $m$ pairwise internally disjoint paths in $G$ from $u$ to $v$.
Let $Y=(N(u) \cup N(v)) \cap \Hub(G)$ and let $G'=G \setminus Y$. Then $|Y| \leq 2d$, and so in $G'$ there are $k$  pairwise internally vertex disjoint paths from $u$ to $v$. But $N_{G'}(u) \cap \Hub(G')=N_{G'}(v) \cap \Hub(G')=\emptyset$, contrary
to Theorem \ref{banana}. This proves Lemma \ref{noncoop}.
\end{proof}

In the theory of tree decompositions, working with  the torso of a bag is a natural way to focus on the bag while still keeping track of its relation to the rest of the graph. Unfortunately,  taking torsos is not an operation closed under induced subgraphs, and so this method does not work when general hereditary classes are considered.  The goal of the next theorem is to design a tool that will allow us to construct a safe alternative to torsos.
   
  \begin{theorem}
    \label{connectors}
Assume that $G$ does not admit a clique cutset. For every $t_1 \in N_T(t_0)$,
        there exists an induced subgraph
         $Conn(t_1)$ of $G_{t_0 \rightarrow t_1}$ such that
        \begin{enumerate}
        \item  We have  $\chi(t_1) \cap \chi(t_0)  \subseteq Conn(t_1)$.
          \label{C-1}
                \item $Conn(t_1) \setminus \chi(t_0)$ is connected and
                  $N(Conn(t_1) \setminus \chi(t_0))=\chi(t_0) \cap \chi(t_1)$.
\label{C-2}
                  %  \item $Conn(t) \setminus \chi(t_0) \neq \emptyset$.
\item No vertex of $Conn(t_1) \setminus \chi(t_0)$ is a hub in the graph
  $(G \setminus G_{t_0 \rightarrow t_1})  \cup Conn(t_1)$.
\label{C-3}
%\item For every $v \in \chi(t_0) \cap \chi(t)$,
%  $v$ has fewer than  $m$ neighbors in $Conn(t) \setminus \chi(t_0)$.
%\label{C-4}
%\item For every $v \in Conn(t) \setminus \chi(t_0)$, $v$ has
%  at most $m$ neighbors in the graph  $(G \setminus G_{t_0 \rightarrow t})  \cup Co%nn(t)$.
%  \label{C-5}
        \end{enumerate}
  \end{theorem}

  \begin{proof}

    Write
    $\chi(t_0) \cap \chi(t_1)=M$.
        Let $K$ be a minimal connected induced subgraph of $G_{t_0 \rightarrow t_1} \setminus \chi(t_0)$ such that 
        $\chi(t_0) \cap \chi(t_1)= N(K)$ (such $K$ exists since $(T, \chi)$ is tight).
                Let $M=\{m_1, \dots, m_s\}$. Since $G$ does not admit a clique cutset, we have that $s \geq 2$.

 We claim that no vertex of $K$ is a hub of
      $(G \setminus G_{t_0 \rightarrow t_1})  \cup (K \cup M)$.
    Let $v \in K$ and let $(H,v)$ be a proper wheel in
    $G' = (G \setminus G_{t_0 \rightarrow t_1}) \cup (K \cup M)$.
Then $(H,v)$ is a proper wheel in $G$.
    Since $v \in K$, it follows that
    $N_G(v) \subseteq G_{t_0 \rightarrow t_1}$.
    By Theorem \ref{connectedbranches}, 
    $G \setminus G_{t_0 \rightarrow t_1}$ is connected, and
    since $(T, \chi)$ is tight,
    every vertex of $M$ has a neighbor  in $G \setminus  G_{t_0 \rightarrow t_1}$.
    It follows that $G \setminus G_{t_0 \rightarrow t_1}$ is
    contained in a    component $D$ of $G \setminus N[v]$, and $M \subseteq N_G[D]$. Theorem \ref{wheelstarcutset} implies that $H \not\subseteq N_G[D]$. Since $H \subseteq G' \setminus \{v\} = (G \setminus G_{t_0 \rightarrow t_1}) \cup M \cup (K \setminus \{v\}) \subseteq N_G[D] \cup (K \setminus \{v\})$, it follows that $H$ contains a vertex in $(K \setminus \{v\}) \setminus N_G[D]$. Let $K' = \{v\} \cup (K \cap N_G[D])$. Then $K' \subsetneq K$, and so from the minimality of $K$, it follows that either $K'$ is not connected, or $M \not\subseteq N(K')$. Suppose first that the latter happens. Then there is an $m_i \in M \setminus N(K')$. It follows that $m_i \not\in N(v)$ as $v \in K'$, and therefore $m_i \in D$. Let $k_i$ be a neighbor of $m_i$ in $K$. Then $k_i \in N[D]$, and so
    $k_i \in K'$, a contradiction; this shows that $M \subseteq N(K')$. It follows that $K'$ is not connected. Since $K$ is connected, it follows that for every $x \in K$, there is a path $P_x$ from $x$ to $v$ with $P_x \subseteq K$. Moreover, since
    $D$ is a component of $G \setminus N[v]$, 
    if $x \in N[D]$, then $P_x \subseteq (N[D] \cap K) \cup \{v\}$. Therefore, for every $x \in K'$, we have $P_x \subseteq K'$; so $K'$ is connected. This is a contradiction, and so the claim that no vertex of $K$ is a hub of
      $(G \setminus G_{t_0 \rightarrow t_1})  \cup (K \cup M)$ follows.

Setting $Conn(t_1)=K \cup M$, Theorem \ref{connectors} follows.
  \end{proof}

  From now on we assume that $G$ does not admit a clique cutset and  construct a graph that is a safe alternative for $\hat \chi(t_0)$. Let $S'$ be a stable set of hubs of $G$, and assume that every $s \in S'$ is $d$-safe. Let $S_{bad}$ denote the set of all vertices in $S'$ that are not $t_0$-cooperative; by Lemma \ref{noncoop}, we have that $|S_{bad}| \leq 1$. Let $S=S' \setminus S_{bad}$ and set
  $$\beta(S')=\left(\chi(t_0) \cup \bigcup_{t_1 \in N_T(t_0)}Conn(t_1)\right) \setminus S_{bad} .$$
  Write $\beta=\beta(S')$.
It follows that for every $t_1 \in N_T(t_0)$, we have that $\beta \subseteq (G \setminus G_{t_0 \rightarrow t_1}) \cup Conn(t_1)$.

 From now on,  assume that $\beta$ does not admit a balanced separator of
  size at most $2m(m-1)+\gamma(t)+1$ where $\gamma(t)$ is as in Theorem \ref{noblocksmalltw_Ct}.
Let us say that a vertex $v$ is {\em unbalanced} if there is a component
$D$ of $\beta  \setminus N[v]$ such that $|D| > \frac{|\beta|}{2}$.

Recall the following well-known lemma:
\begin{lemma}[\cite{cygan}]
\label{lemma:tw-to-weighted-separator}
Let $G$ be a graph and let $k$ be a positive integer. If $\tw(G) \leq k$, then $G$ has a $(w, c)$-balanced separator of size at most $k+1$ for every $c \in [\frac{1}{2}, 1)$ and for every weight function $w:V(G) \to [0, 1]$ with $w(G) = 1$.
\end{lemma}

We prove:

  \begin{lemma}
    \label{smallC}
    %Suppose that $G$ does not admit a clique cutset and
    Let $s \in S \cap \chi(t_0)$. Then the following hold. 
    \begin{enumerate}
\item      $|N_{\hat{\chi}(t_0)}(s)| <  2m(m-1)$.
\item   $|N_{\chi(t_0)}(s)| < 2m(m-1)$.
  \item The vertex $s$ is unbalanced (in $\beta$).
\end{enumerate}
  \end{lemma}

  \begin{proof}
    Let $s \in S$.
    Since $s$ is $t_0$-cooperative, we have that
    $| N_{\hat{\chi}(t_0)}(s)| < 2m(m-1)$, and the first assertion of the
    lemma holds. Since $\chi(t_0)$ is a subgraph of $\hat{\chi}(t_0)$,
    the second assertion follows immediately from the first.

    We now prove the third assertion. Let $\delta(s)$ be the set of
    all vertices 
    $t_1 \in N_T(t_0)$  such that $s \in \chi(t_0) \cap \chi(t_1)$.
    Write
    $\Delta(s)=  N_{\chi(t_0)}(s) \cup \bigcup_{t_1 \in \delta(s)}(\chi(t_1) \cap \chi(t_0))$.  Since $\Delta(s) \subseteq N_{\hat{\chi}(t_0)}[s]$, by the first assertion of the theorem, we have that  $|\Delta(s)| \leq 2m(m-1)$. Since $\beta$ does not admit a
    $2m(m-1)$-balanced separator, it follows that there is a component
    $D'$ of $\beta \setminus \Delta(s)$ with $|D'| > \frac{|\beta|}{2}$.

   We claim that  $D' \cap Conn(t')  = \emptyset$ for every $t' \in \delta(s)$.
    Suppose not, and let $t' \in \delta(s)$ be such that
    $D' \cap Conn(t')  \neq \emptyset$. Since $D'$ is a component of
    $\beta \setminus \Delta(s)$, it follows that $D'=Conn(t') \setminus \chi(t_0)$. By Theorem \ref{connectors}\eqref{C-3}, we have that $\Hub(D')=\emptyset$. Now by
    Theorem \ref{noblocksmalltw_Ct} and Lemma~\ref{lemma:tw-to-weighted-separator}, we conclude that $D'$ has a balanced separator $X$ of size
    $\gamma(t)+1$. But now $X \cup \Delta(s)$ is a balanced separator
    of $\beta$ of size $2m(m-1)+\gamma(t)+1$, a contradiction.
    This proves the claim that  $D' \cap Conn(t')  = \emptyset$ for every $t' \in \delta(s)$.

    \medskip
    
    Since $N_{\beta}(s) \setminus \Delta(s) \subseteq \bigcup_{t_1 \in \delta(s)}Conn(t_1)$, the claim of the previous paragraph implies that  $D' \cap N_{\beta}(s)=\emptyset$.
    Let $D$ be a component of $\beta  \setminus N_{\beta}[s]$ such that
    $D' \subseteq D$; then $|D| > \frac{|\beta|}{2}$, as required. This proves
    that $s$ is unbalanced and completes the proof of Lemma \ref{smallC}.
  \end{proof}

% \begin{lemma} \label{lem:unbalanced}
%     Every vertex of $S \cap \beta$ is unbalanced. \check{not sure if we need this; added it since this statement was given without proof}
% \end{lemma}
%    \begin{proof}
%        By Lemma \ref{smallC}, it suffices to consider vertices $s$ in $\beta \setminus \chi(t_0)$; so $s \in Conn(t) \setminus \chi(T_0)$ for some $t \in N_T(t_0)$. It follows that $N(s) \subseteq Conn(t)$. Let $X = \chi(t) \cap \chi(t_0)$. Since $|X| \leq m-1$, it follows that $X$ is not a balanced separator in $\beta$. Therefore, $\beta \setminus X$ contains a component $D$ with $|D| > \frac{|\beta|}{2}$. If $D \cap Conn(t) = \emptyset$, then there is a component $D'$ of $\beta \setminus N[s]$ containing $D$ and $|D'| > \frac{|\beta|}{2}$; so $s$ is unbalanced as desired. Therefore, we have that $D \cap Conn(t) \neq \emptyset$. 

%        Since $X$ separates $Conn(t) \setminus X$ from $\beta \setminus (X \cup Conn(t))$, we may assume that $D \subseteq Conn(t) \setminus X = Conn(t) \setminus \chi(t_0)$. Now, as before, by Theorem \ref{noblocksmalltw_Ct}, we conclude that $D$ has a balanced separator $Y$ of size $\gamma(t) + 1$. But then $X \cup Y$ is a balanced separator in $\beta$ of size at most  $m-1 + \gamma(t) + 1$, a contradiction. 
 %   \end{proof}
 
  As we mentioned in Section~\ref{intro},
  handling balanced vertices was  somewhat technical in \cite{TWIII};  it
  required constructing  an artificial auxiliary graph. Using leanness provides a
  much more natural framework to deal with this obstacle.

  A {\em separation} of a graph $G$ is a triple $(Y,X,Z)$ of pairwise disjoint
  subsets of $V(G)$ with $X \cup Y \cup Z = V(G)$ such that $Y$ is anticomplete to
  $Z$.
  As in earlier papers in this series, we associate a certain unique star separation to every  vertex of $S \cap \chi(t_0)$. For $v \in S \cap \chi(t_0)$,
  let
    $B(v)$ be the unique component of $\beta \setminus N[v]$ with
  $|B(v)| > \frac{|\beta|}{2}$, let $C(v)=N(B(v)) \cup \{v\}$, and  finally, let $A(v)=\beta \setminus (B(v) \cup C(v))$. Then $(A(v), C(v),B(v))$ is the 
{\em canonical star separation of $\beta$ corresponding to $v$}. We show:
\begin{lemma} \label{nohub}
  The vertex $v$ is not a hub of $\beta \setminus A(v)$.
    \end{lemma}

    \begin{proof}
      Suppose
      %first
      that $(H,v)$ is a proper wheel in $\beta \setminus A(v)$.
      Then $H \subseteq N[B(v)]$, contrary to Theorem \ref{wheelstarcutset}.
      %Next suppose that $(\Sigma,P)$ is a loaded pyramid in $G \setminus A(v)$;
    %with the notation as in the defintion of a loaded pyramid let $v=b_2$ be the loaded pyramid corner of $(\Sigma, P)$. Then $a,b_1,b_3 \in N[B(v)]$, contrary 
%to \ref{loadcorener}.
      This proves Lemma~\ref{nohub}.
    \end{proof}

Let $\mathcal{O}$ be a linear order on $S \cap \chi(t_0)$.  Following \cite{TWIII}, we say
     that two unbalanced vertices $u,v$ of $S \cap \chi(t_0)$ are {\em star twins} if $B(u)= B(v)$, $C(u) \setminus \{u\} = C(v) \setminus \{v\}$, and
     $A(u) \cup \{u\} = A(v) \cup \{v\}$.
     %(Note that every two of these conditions imply the third.)

Let $\leq_A$ be a relation on $S \cap \chi(t_0)$ defined as follows: 
\begin{equation*}
\hspace{2.5cm}
x \leq_A y \ \ \ \text{ if} \ \ \  
\begin{cases} x = y, \text{ or} \\ 
\text{$x$ and $y$ are star twins and $\mathcal{O}(x) < \mathcal{O}(y)$, or}\\ 
\text{$x$ and $y$ are not star twins and } y \in A(x).\\
\end{cases}
\end{equation*} 
Note that if $x \leq_A y$, then either $x = y$, or $y \in A(x).$  
We need the following result from \cite{TWIII}:

\begin{lemma} [Abrishami,  Chudnovsky, Hajebi, Spirkl \cite{TWIII}]
  \label{A_order}
      The relation $\leq_A$ is a partial order on $S \cap \chi(t_0)$. 
\end{lemma}

In view of Lemma \ref{A_order}, 
let $\Core(S')$ be  the set of all $\leq_A$-minimal elements of $S \cap \chi(t_0)$. Define
$$\beta^A(S')=\bigcap_{v \in \Core(S')} (B(v) \cup C(v)).$$
The following was also proved in \cite{TWIII}:

\begin{lemma}
 [Abrishami,  Chudnovsky, Hajebi, Spirkl \cite{TWIII}]
  \label{looselylaminar}
    Let $u,v \in \Core(S')$. Then
    $A(u) \cap C(v)=C(u) \cap A(v)=\emptyset$.
\end{lemma}

Next, we need an analogue of a result of \cite{TWIII}, summarizing the behavior
of $\beta^A(S')$.

\begin{theorem} 
  \label{A_centralbag}
    The following hold:
    \begin{enumerate}
    \item For every $v \in \Core(S')$, we have $C(v) \subseteq \beta^{A}(S')$. \label{A-1}
\item $|C(v) \cap \chi(t_0)| \leq 2m(m-1)$ for every $v \in \Core(S')$.
        \label{A-2}
      \item For every component $D$ of $\beta  \setminus \beta^{A}(S')$, there exists $v \in \Core(S')$ such that $D \subseteq A(v)$. Further, if $D$ is a component of $\beta \setminus \beta^A(S')$ and $v \in \Core(S')$ such that $D \subseteq A(v)$, then $N_{\beta}(D) \subseteq C(v)$. \label{A-3}
      \item $S' \cap \Hub(\beta^A(S'))=\emptyset$. \label{A-4}
    \end{enumerate}
\end{theorem}

\begin{proof}
    \eqref{A-1} is immediate
    from Lemma  \ref{looselylaminar}, and \eqref{A-2} follows 
    from Lemma \ref{smallC}.
    
Next we prove  \eqref{A-3}. Let $D$ be a component of
$\beta \setminus \beta^A(S')$. Since  $\beta \setminus \beta^A(S')=\bigcup_{v \in \Core(S')}A(v)$, there exists
    $v \in \Core(S')$ such that $D \cap A(v) \neq \emptyset$.
    If $D \setminus A(v) \neq \emptyset$, then, since $D$ is connected, it follows that $D \cap N(A(v)) \neq \emptyset$; but then $D \cap C(v) \neq \emptyset$, contrary to \eqref{A-1}. Since $N_{\beta}(D) \subseteq \beta^A(S')$ and $N(D) \subseteq A(v) \cup C(v)$, it follows that $N_{\beta} (D) \subseteq C(v)$. 
    This proves \eqref{A-3}.

    To prove \eqref{A-4}, let $u \in S' \cap \Hub(\beta^A(S'))$.
    By Theorem \ref{connectors}\eqref{C-3}, we have that $u \in \chi(t_0)$.
    Since $\beta^A(S') \subseteq \beta$, we deduce that $u \not \in S_{bad}$.
    By  Lemma \ref{nohub}, it follows that
$\beta^A(S') \not \subseteq B(u) \cup C(u)$, and therefore $u \not \in \Core(S')$.
But then $u \in A(v)$ for some $v \in \Core(S')$, and so $u \not \in \beta^A(S')$,
a contradiction. This proves \eqref{A-4} and completes the proof of
 Theorem \ref{A_centralbag}.
  \end{proof}

In the course of the proof of Theorem \ref{main}, we will inductively obtain a bound on
$\tw(\beta^A(S))$. Next we explain how  to transform a tree decomposition of
$\beta^A(S')$ into a  tree decomposition of $\beta$. 

Let $D_1, \dots, D_s$ be the components of
    $\beta \setminus \beta^A(S')$. For $i \in \{1,\dots, s\}$, let
    $r(D_i)$ be the  $\mathcal{O}$-minimal vertex $v$ of $\Core(S')$ such that
    $D_i \subseteq A(v)$ (such $v$ exists by  Theorem \ref{A_centralbag}\eqref{A-3}).
Write $\beta_0=\chi(t_0)$.

    Let $(T_0,\chi_0)$ be a
    tree decomposition of $\beta^A(S')$, and for $i \in \{1, \dots, s\}$
    let $(T_i,\chi_i)$ be a tree decomposition of $D_i$.
Let $T_\beta$ be the tree obtained from the union of $T_0,T_1, \dots, T_s$
    as follows. Let  $i \in \{1, \dots, s\}$. 
    Choose  $t \in T_0$
    such that $r(D_i) \in \chi_0(t)$  and
    add an edge from $t$ to an arbitrarily chosen vertex of $T_i$.
    For $u \in T_{\beta}$, let $\chi_{\beta}(u)$ be defined as follows.
\begin{itemize}
    \item  If $u \in V(T_0)$, let
    $$\chi_{\beta}(u)=\chi_0(u)  \cup \bigcup_{v \in \Core(S') \cap \chi_0(u)} C(v) .$$
    \item If $u \in V(T_i)$ for $i \in \{1, \dots, s\}$, let
    $$\chi_{\beta}(u)=\chi_i(u)  \cup C(r(D_i)).$$
\end{itemize}

    \begin{theorem}
      \label{A_extendtree}
      With the notation as above, $(T_{\beta},\chi_{\beta})$ is a tree decomposition of
      $\beta$. Moreover,
      \begin{itemize}
     \item  for every $t \in T_0$,
      $$|\chi_{\beta}(t) \cap \beta_0| \leq |\chi_0(t)|+2m(m-1)(|\chi_0(t) \cap \Core(S')|),$$
       and
       \item for every $i \in \{1, \dots, s\}$ and
      $t \in T_i$,
      $$|\chi_{\beta}(t)  \cap \beta_0| \leq |\chi_i(t)|+2m(m-1).$$
\end{itemize}
      \end{theorem}      

    \begin{proof}
      The fact that $(T_\beta, \chi_\beta)$ is a tree decomposition of $\beta$ follows from Theorem  6.6 of \cite{TWIII}. Recall that by          Lemma \ref{smallC},
      $|(C(v) \setminus v) \cap \beta_0| < 2m(m-1)$
      for every $s \in \Core(S')$.
      Now the theorem follows directly from 
      the definition of $(T_{\beta}, \chi_{\beta})$.
          \end{proof}

  We finish this section with a theorem that allows us to transform
  a tree decomposition of $\beta$ into a tree decomposition of $G$.

  Write $N_T(t_0)=\{t_1, \dots, t_r\}$, and let $D_i'=G_{t_0 \rightarrow t_i} \setminus \chi(t_0)$ for $i \in \{1, \dots, r\}.$

    Let $(T_0',\chi_0')$ be a
    tree decomposition of $\beta$, and for $i \in \{1, \dots, r\}$
    let $(T_i',\chi_i')$ be a tree decomposition of $D_i'$.
    Let $U$ be the tree obtained from the union of $T_0',T_1', \dots, T_r'$
    as follows. Let  $i \in \{1, \dots, r\}$. 
    Choose  $t \in T_0'$
    such that $\chi_0'(t) \cap (Conn(t_i) \setminus \chi(t_0)) \neq \emptyset$ and
    add an edge from $t$ to an arbitrarily chosen vertex of $T_i'$.

    For $u \in U$, let $\psi(u)$ be defined as follows.
\begin{itemize}
    \item  If $u \in V(T_0')$, let
      $$\psi(u)=S_{bad} \cup (\chi_0'(u)  \cap \beta_0) \cup
      \bigcup_{t_i \text{ s.t. }   \chi_0'(u) \cap (Conn(t_i) \setminus \chi(t_0)) \neq \emptyset} (\beta_0 \cap \chi(t_i)).$$
      
    \item If $u \in V(T_i')$ for $i \in \{1, \dots, r\}$, let
    $$\psi(u)=S_{bad} \cup \chi_i'(u) \cup (\beta_0 \cap \chi(t_i)).$$
\end{itemize}

    \begin{theorem}
      \label{delta_extendtree}
      With the notation as above, $(U,\psi)$ is a tree decomposition of $G$.

    \end{theorem}

   \begin{proof}
     Since $U$ is obtained by adding a single edge from $T_0'$ to each of
     the trees $T_1', \dots, T_r'$, it follows that $U$ is a tree.
     Clearly, every vertex of $G$ is in $\psi(v)$ for some $v \in V(U)$.
     Next we check that for every edge $xy$ of $G$, there exists $v \in V(U)$
     such that $x,y \in \psi(v)$. This is clear if  $x,y \in \chi(t_0)$ or if
     $x,y \in D_i'$ for some $i$; thus we may assume that $x \in \chi(t_0)$ and
     $y \in D_1'$, say.    Since $(T, \chi)$ is a tree decomposition of $G$,
     it follows that  $x \in \chi(t_0) \cap \chi(t_1)$. 
     Let $v \in V(T_1')$ such that $y \in \chi_1'(v) \subseteq \psi(v)$; then $x,y \in \psi(v)$ as required.

 Let $x \in V(G)$. 
          For $i \in \{0, 1, \dots, r\}$, define $F_{i}(x)=\{v \in V(T_i') \text{ such that }x \in \psi(v)\}$. Let $F(x)=\bigcup_{i=0}^rF_i(x)$. We  need to show that for every $x \in V(G)$, the induced subgraph
     $U[F(x)]$ is connected. This is true if $x \in S_{bad}$, since then $F(x) = V(U)$. Therefore, we now assume that $x \not\in S_{bad}$.      If $x \not\in \beta_0$, then $F_0(x)=\emptyset$, and there exists a unique $i$ such that $x \in D_i'$; therefore $F(x)=F_i(x)$, and $U[F(x)]$ is connected because $T_i'[F_i(x)]$ is connected.
     Thus we may assume that $x \in \beta_0$.
     Let $I \subseteq \{1, \dots, r\}$
     be the set of all $i$ such that $x \in \chi(t_0) \cap \chi(t_i)$. It follows that $F_i(x) = \emptyset$ for all $i \in \{1, \dots, r\} \setminus I$, and $F_i(x) = V(T_i')$ for all $i \in I$.

     First we show that for every $i \in I$,  there is an edge between
     $T_i'[F_i(x)] = T_i'$ and     $T_0'[F_0(x)]$. 
     Let $i \in I$;
     and let $v \in T_0'$ be such that $v$ is adjacent to a vertex of
     $T_i'$. Then $(Conn(t_i) \setminus \chi(t_0)) \cap  \chi_0'(v) \neq \emptyset$; consequently $\chi(t_0) \cap \chi(t_i) \subseteq \psi(v)$, and therefore, $x \in \psi(v)$. Thus there is an edge between $T_i'[F_i(x)]$ and
     $T_0'[F_0(x)]$, as required.

     Now to show that $U[F(x)]$ is connected, it is enough to prove that
     $T_0'[F_0(x)]$ is connected.
     %Let $J(x)=\{t \in N_{T}(t_0) \text { such that } x \in Conn(t)\}$.
Write $\chi_0'^{-1}(x) := \{v \in T'_0 : x \in \chi_0'(v)\}$ and observe that
$$F_0(x)=\chi_0'^{-1}(x) \cup \bigcup_{t \in I} \chi_0'^{-1}(Conn(t) \setminus \chi(t_0)).$$
Since by Theorem \ref{connectors}\eqref{C-2}, we have that $\{x\} \cup (Conn(t) \setminus \chi(t_0))$ is connected, it follows from basic properties of a tree decomposition that $T_0'[\chi_0'^{-1}(\{x\} \cup (Conn(t) \setminus \chi(t_0)))]$ is connected. Since $\chi_0'^{-1}(\{x\}) \neq \emptyset$, it follows that the union $T_0'\left[\bigcup_{t \in I} \chi_0'^{-1}(\{x\} \cup (Conn(t) \setminus \chi(t_0)))\right]$ is connected. We deduce 
       that $T_0'[F_0(x)]$ is connected, as required. This
       proves  Theorem \ref{delta_extendtree}.
      
     \end{proof}

   \section{Connectifiers  revisited}
   \label{sec:connectifiers}
   
  We start by recalling a well known theorem of Erd\H{o}s and Szekeres \cite{ES}.
  \begin{theorem}[Erd\H{o}s and Szekeres \cite{ES}]
    \label{ESz}
Let $x_1, \dots, x_{n^2+1}$ be a sequence of distinct reals. Then
there exists either an increasing or a decreasing $(n + 1)$-sub-sequence.
\end{theorem}

  We will also need the following easy lemma, whose proof we include for completeness.

  \begin{lemma}
    \label{digraph}
    Let  $k$ be an integer and let $D$ be a directed graph  in which  every vertex has at most  $k$ outneighbors. Let $D^-$ be the underlying undirected graph of $D$.
    Then there exists $X \subseteq V(D)$ with $|X| \geq \frac{|V(D)|}{2k+1}$ such that
    $X$ is a stable set of $D^-$.
  \end{lemma}

  \begin{proof}
  Every induced subgraph $H$ of $D^{-}$ has at most $k|V(H)|$ edges, and therefore contains a vertex of degree at most $2k$. This shows that $D^{-}$ is $2k$-degenerate, and therefore $(2k+1)$-colorable. So $D^{-}$ has a stable set of size at least $|V(D)|/(2k+1)$.     This proves Lemma \ref{digraph}.
    \end{proof}

  A vertex $v$ in a graph $G$ is said to be a \textit{branch vertex} if $v$
  has degree more than two. By a {\em caterpillar} we mean a tree $C$ with
  maximum degree three such that  there exists  a path $P$ of $C$ where all
  branch vertices of $C$ belong to $P$.
We call a maximal such path $P$ the {\em spine} of $P$. % Note (Sophie): changed this to a maximal path because otherwise the definition of legs in the case of line graphs of caterpillars is a mess
  (Our definition of a caterpillar is
  non-standard for two reasons: a caterpillar is often allowed to be of
  arbitrary maximum degree, and a spine often contains all vertices of degree
  more than one.) 
   By a \textit{subdivided star} we mean a graph isomorphic to
  a subdivision of the complete bipartite graph $K_{1,\delta}$ for some $\delta\geq 3$.
  In other words, a subdivided star is a tree with exactly one branch vertex,
  which we call its \textit{root}. For a graph $H$, a vertex $v$ of $H$ is
  said to be \textit{simplicial} if $N_H(v)$ is a clique. We denote by
  $\mathcal{Z}(H)$ the set of all simplicial vertices of $H$. Note that for
  every tree $T$, $\mathcal{Z}(T)$ is the set of all leaves of $T$. An edge
  $e$ of a tree $T$ is said to be a \textit{leaf-edge} of $T$ if $e$ is
  incident with a leaf of $T$. It follows that if $H$ is the line graph of a
  tree $T$, then $\mathcal{Z}(H)$ is the set of all vertices in $H$
  corresponding to the leaf-edges of $T$. The following is proved in
  \cite{TWVII} based on (and refining) a result from \cite{Davies}. 
\begin{theorem}[Abrishami, Alecu, Chudnovsky, Hajebi, Spirkl \cite{TWVII}]\label{connectifier1}
  For every integer $h\geq 1$, there exists an integer $\mu=\mu(h)\geq 1$ with the following property. Let $G$ be a connected graph with no clique of cardinality $h$ and let
  $S \subseteq G$ such that $|S|\geq \mu$. Then either some path in $G$ contains $h$ vertices of $S$, or 
  there is an induced subgraph $H$ of $G$ with $|H\cap S|=h$ for which one of the following holds.
  \begin{itemize}
  \item $H$ is either a caterpillar or the line graph of a caterpillar with $H\cap S=\mathcal{Z}(H)$.
  \item $H$ is a subdivided star with root $r$ such that $\mathcal{Z}(H)\subseteq H\cap S\subseteq \mathcal{Z}(H)\cup \{r\}$.
  \end{itemize}
\end{theorem}

Let $H$ be a graph that is either a path, or a caterpillar, or the line graph of a caterpillar,  or a subdivided star with root $r$.
We define an induced subgraph of $H$, denoted by $P(H)$, which we will use
throughout the paper.
If $H$ is a path let $P(H)=H$.
If $H$ is a caterpillar, let $P(H)$ be the spine of $H$.
  If $H$ is the line graph of a caterpillar $C$, let $P(H)$ be the path of $H$
  consisting of the vertices of $H$ that correspond to the edges of the spine of $C$. If $H$ is a subdivided claw with root $r$, let $P(H)=\{r\}$.
  The {\em legs} of $H$ are the components of $H \setminus P$.
Our first goal is to prove the following variant of Theorem \ref{connectifier1}:

\begin{theorem}
  \label{connectifier2}
  For every integer $h'\geq 1$, there exists an integer $\nu=\nu(h')\geq 1$ with the following property. Let $G$ be a connected graph with no clique of cardinality $h'$. Let $S' \subseteq G$ such that  $|S'|\geq \nu$, $G \setminus S'$ is connected  and every vertex of $S'$ has a neighbor in $G \setminus S'$.
  Then  there is a set $\tilde S \subseteq S'$ with $|\tilde S|=h'$ and an induced subgraph $H'$ of $G \setminus S'$ for which one of the following holds.
  \begin{itemize}
    \item $H'$ is a path and every vertex of $\tilde S$ has a neighbor in $H'$; or    \item $H'$ is a caterpillar, or the line graph of a caterpillar, or a subdivided star. Moreover, every vertex of $\tilde S$ has a unique neighbor in $H'$ and  
      $H'\cap N(\tilde S)=\mathcal{Z}(H')$.
          \end{itemize}
\end{theorem}

\begin{proof}
    Write $G' := G \setminus S'$, let $h = 1 + h' + 2h'^2$, and let $\nu = h' \mu(h)$, where $\mu$ is as in Theorem~\ref{connectifier1}. Assume that the first bullet point above does not hold, that is:
    
    \sta{\label{pathneighbors} For every path $Q$ of $G'$, $|N(Q) \cap S'| \leq h' - 1$.}   

    Now choose, for every $s \in S'$, a neighbor $n(s)$ of $s$ in $G'$, and let $S := \{n(s) : s \in S'\}$. By \eqref{pathneighbors}, $|n^{-1}(v)| \leq h' - 1$ for all $v \in G'$ (and in particular for all $v \in S$), and so $|S| \geq |S'|/h' = \mu(h)$. Let $S'' \subseteq S'$ be minimal such that $S = \{n(s) : s \in S''\}$. It follows that $n$ is a bijection between $S''$ and $S$.%, and in particular, $|S''| \geq \mu(h)$.%, and that each $v \in S$ equals $n(s)$ for a unique $s \in S''$ (that is, $n$ is a bijection between $S$ and $S''$).    

    We now apply Theorem~\ref{connectifier1} to $G', S$ and $h$. By \eqref{pathneighbors}, the path outcome of the theorem does not happen, so there is an induced subgraph $H$ of $G'$ with $|H \cap S| = h$, for which one of the following holds:

    \begin{itemize}
        \item $H$ is either a caterpillar or the line graph of a caterpillar with $H \cap S = \mathcal Z(H)$.

        \item $H$ is a subdivided star with root $r$ such that $\mathcal Z(H) \subseteq H \cap S \subseteq \mathcal Z(H) \cup \{r\}$. 
    \end{itemize}

    Let $P := P(H)$, and let $A := S'' \cap n^{-1}(\mathcal Z(H) \setminus P)$; in other words, $A$ consists of the vertices of $S''$ whose selected neighbors are simplicial vertices of $H$ other than the endpoint(s) of $P$. In particular, $|A| \geq h - 2 = 2h'^2 + h' - 1$. 
    Moreover, for each $v \in A$, $n(v)$ belongs to a unique leg of $H$, which yields a bijective correspondence between $A$ and the set of legs of $H$ (henceforth, if $n(v) \in D$ for a vertex $v \in A$ and a leg $D$ of $H$, we will say that $v$ {\em corresponds} to $D$, and vice-versa).  

    Let $A'' \subseteq A$ be the set of vertices of $A$ anticomplete to $P$. We note that, by \eqref{pathneighbors}, $|A''| \geq |A| - (h' - 1) \geq 2h'^2$. Let $H''$ be obtained from $H$ by removing all legs corresponding to vertices in $A \setminus A''$. In particular, we note that $H''$ is still a caterpillar, line graph of a caterpillar, or star, according to what $H$ was, and that the vertices of $A''$ correspond bijectively to the legs of $H''$.
    
    Define now a directed graph $F$ as follows. $V(F)$ is the set of legs of $H''$, and we have an arc from $D_1 \in F$ to $D_2 \in F$ if the vertex $s_2$ corresponding to $D_2$ has a neighbor in $D_1$. By \eqref{pathneighbors}, every vertex of $F$ has at most $h' - 1$ outneighbours, and so by Lemma~\ref{digraph}, the underlying undirected graph of $F$ contains a stable set $T$ of size $\frac{|V(F)|}{2(h' - 1) + 1} \geq \frac{2h'^2}{2h' - 1} \geq h'$. Let $\tilde{S}$ be the set corresponding to the legs in $T$, and let $H'$ be obtained from $H''$ by deleting the legs not in $T$. It is routine to check that $H'$ and $\tilde{S}$ satisfy the second outcome of Theorem~\ref{connectifier2}.
\end{proof}

Next we introduce more terminology.
Let $G$ be a graph, let $P=p_1 \dd \dots \dd p_n$ be a path of $G$ and let
$X=\{x_1, \dots, x_k\}  \subseteq G \setminus P$. We say that $(P,X)$ is an
{\em alignment} if 
$N_P(x_1)=\{p_1\}$, $N_P(x_k)=\{p_n\}$, 
every vertex of $X$  has a neighbor in $P$, and there exist
$1< j_2 < \dots< j_{k-1} < j_k =  n$ such that
$N_P(x_i) \subseteq p_{j_i} \dd P \dd p_{j_{i+1}-1}$ for $i \in \{2, \dots, k-1\}$.
We also say that $x_1, \dots, x_k$ is {\em the order on $X$ given by the
  alignment $(P,X)$}.
An alignment $(P,X)$ is {\em wide}  if 
each of $x_2, \dots, x_{k-1}$ has two non-adjacent neighbors in  $P$,
{\em spiky} if  each of $x_2, \dots, x_{k-1}$ has a unique neighbor in $P$ and
{\em triangular} if 
each of $x_2, \dots, x_{k-1}$ has exactly two neighbors in $P$  and they are
adjacent. An alignment is {\em consistent} if it is
wide, spiky or triangular.
Next, let $H$ be a caterpillar or the line graph of a caterpillar 
and let $S$ be a set of vertices disjoint from $H$ such that
every vertex of $S$ has a unique neighbor in $H$ and  
      $H\cap N(S)=\mathcal{Z}(S)$.
Let $X$ be the set of vertices of $H \setminus P(H)$ that have neighbors in
$P(H)$. Then the neighbors of elements of $X$ appear in $P(H)$ in order (in fact, $(X, Q)$ is an alignment for some subpath $Q$ of $P(H))$;
let $x_1, \dots, x_k$ be the corresponding order
on $X$. Now, order the vertices of $S$ as
$s_0, s_1, \dots, s_k, s_{k+1}$ where $s_i$ has a neighbor in the leg of $H$ containing
$x_i$ for $i \in \{1, \dots, k\}$, and $s_0, s_{k+1}$ are the ends of $P(H)$ in this order. We say that $s_0, s_1, \dots, s_k, s_{k+1}$ is {\em the order on $S$ given by $(H,S)$}.

Next, let  $H$ be   an induced subgraph of $G$ that is  a caterpillar, or the line graph of a caterpillar, or a subdivided star and let $X \subseteq G \setminus H$ be such that every vertex of $X$ has a unique neighbor in $H$ and  $H \cap N(X)=\mathcal{Z}(H)$. We say that $(H,X)$ is a {\em consistent connectifier} and it is  {\em spiky, triangular, or stellar} respectively.

Our next goal is, starting with a graph $G \in \mathcal{C}_t$ and a
stable set $X \subseteq V(G)$, to produce certain consistent connectifiers. We start with a lemma.

\begin{lemma}
  \label{twosides_lemma}
    Let $G \in \mathcal{C}$ and assume that $V(G)=H_1 \cup H_2 \cup X$
    where $X$ is a stable set with $|X| \geq 3$ and $X \cap \Hub(G)=\emptyset$.
    Suppose that  for $i \in \{1,2\}$, the pair $(H_i,X)$ is a consistent alignment, or a consistent connectifier.
    Assume also that  if neither of $(H_1,X),(H_2,X)$ is stellar,
    then
  the orders given on $X$ by $(H_1,X)$ and by $(H_2,X)$ are the same.
  Then (possibly switching the roles of $H_1$ and $H_2$), we have that:
  \begin{itemize}
\item
  $(H_1,X)$ is a triangular alignment or a triangular connectifier; and
\item $(H_2,X)$ is a spiky connectifier, a stellar connectifier, a spiky alignment or a wide alignment.
  \end{itemize}
  Moreover, if $(H_1,X)$ is a triangular alignment, then $(H_2,X)$ is not a
  wide alignment.
  \end{lemma}

\begin{proof}
  If at least one $(H_i,X) \subseteq \{(H_1,X), (H_2,X)\}$
  is not a stellar connectifier, we
  let $x_1, \dots, x_k$ be the order given on $X$ by $(H_i,X)$.
  If both of $(H_i,X)$ are  stellar connectifiers, we let $x_1, \dots, x_k$ be an arbitrary order on $X$.
  Let $H$ be the unique hole contained in $H_1 \cup H_2 \cup \{x_1,x_k\}$.
  For $j \in \{1,2\}$ and $i \in \{1, \dots, k\}$, if
  $H_j$ is a connectifier, let
  $D_i^j$ be the leg of $H_j$ containing a neighbor of $x_i$; and
  if $H_j$ is an alignment let $D_i^j=\emptyset$.

  Suppose first that  $(H_1,X)$ is a triangular alignment or a triangular connectifier.
  If $(H_2,X)$ is a triangular alignment or a triangular connectifier, then 
   for every $i \in \{2, \dots, k-1\}$, the graph 
  $H \cup D_i^1 \cup D_i^2 \cup \{x_i\}$ is either a prism or an even wheel with center $x_i$, a contradiction.
  This proves that $(H_2,X)$ is a stellar connectifier, or  a spiky connectifier, or 
  a spiky alignment, or a wide alignment. We may assume that
  $(H_1,X)$ is a triangular alignment and $(H_2,X)$ is a wide alignment, for otherwise the theorem holds. But now
   for every $x \in X \setminus \{x_1,x_k\}$,
  $(H,x)$ is a proper wheel, a contradiction. 

  Thus we may assume that for $i \in \{1,2\}$, the pair 
  $(H_i, X)$ is a a stellar connectifier, or  a spiky connectifier, or 
  a spiky alignment, or a wide alignment.
  Now for every
  $x_i \in X \setminus \{x_1,x_k\}$, the graph
  $H \cup D_i^1 \cup D_i^2 \cup \{x_i\}$ is either a theta  or a proper wheel with center $x$, a contradiction.
This proves Lemma \ref{twosides_lemma}.
  \end{proof}

We now prove the main result of this section, which is a refinement of
 Theorem \ref{connectifier2} for graphs in $\mathcal{C}$.

\begin{theorem}
\label{twosides}
  For every integer $x\geq 1$, there exists an integer $\tau=\tau(x) \geq 1$ with the following property.
    Let $G \in \mathcal{C}_x$ and assume that $V(G)=D_1 \cup D_2 \cup Y$
  where
  \begin{itemize}
  \item         $Y$ is a stable set with $|Y| = \tau$,
  \item $D_1$ and $D_2$ are components of $G \setminus Y$, and 
  \item $N(D_1)=N(D_2)=Y$.
  \end{itemize}
  Assume that $Y \cap \Hub(G)=\emptyset$.
  Then there exist $X \subseteq Y$ with $|X|=x$,
$H_1 \subseteq D_1$ and $H_2 \subseteq D_2$
   (possibly with the roles of $D_1$ and $D_2$ reversed) such that
  \begin{itemize}
    \item   $(H_1,X)$ is  a triangular connectifier or a triangular alignment;
  
  \item   $(H_2,X)$ is a stellar connectifier, or a spiky connectifier, or a
    spiky alignment or a wide alignment; and
    \item  if $(H_1,X)$ is a triangular alignment, then $(H_2,X)$ is not a
  wide alignment.
      \end{itemize}
  Moreover, if neither of $(H_1,X),(H_2,X)$ is a stellar connectifier, then
  the orders given on $X$ by $(H_1,X)$ and by $(H_2,X)$ are the same.
  \end{theorem}

\begin{proof}
  Let $z=(27)^2 \cdot 36 x^4$ and let $\tau(x)=\nu(\nu(z))$, where $\nu$ is as in Theorem \ref{connectifier2}.
  Applying Theorem \ref{connectifier2} twice, we obtain a set $Z \subseteq Y$ with $|Z|=z$
  and $H_i \subseteq D_i$ such that either 
   \begin{itemize}
    \item $H_i$ is a path and every vertex of $Z$ has a neighbor in $H_i$; or
    \item $(H_i,X)$ is  a consistent connectifier
        \end{itemize}
        for every $i \in \{1, 2\}$.

   \sta{\label{align} Let $i \in \{1, 2\}$ and $y \in \mathbb N$. If $H_i$ is a path and every vertex of $Z$ has a neighbor in $H_i$, then either some vertex of $H_i$ has $y$ neighbors in $Z$, or
     there exists $Z' \subseteq Z$ with $|Z'| \geq \frac{|Z|} {27y}$ and a  subpath $H_i'$ of $H_i$ such that $(H_i,Z')$ is a consistent alignment.}

Let $H_i=h_1 \dd \dots \dd h_k$. We may assume that $H_i$ is chosen  minimal satisfying Theorem \ref{connectifier2}, and so there exist
   $z_1,z_k \in Z$ such that $N_{H_i}(z_j)=\{h_j\}$ for $j \in \{1,k\}$. 

   We may assume that $|N_{Z}(h)| <y$ for every $h \in H_i$. Let $Z_1$ be the set of vertices in $Z$ with exactly one neighbor in $H_i$. Then, if $|Z_1| \geq |Z|/3$, it follows that $Z_1$ contains a set $Z'$ with $|Z'| \geq |Z_1|/y \geq |Z|/(3y)$ such that no two vertices in $Z'$ have a common neighbor in $H_i$. We may assume that $z_1, z_k \in Z'$. Therefore, $(H_i, Z')$ is a spiky alignment. 

   Next, let $Z_2$ be the set of vertices in $z \in Z$ such that either $z \in \{z_1, z_k\}$ or  has exactly two neighbors in $H_i$, and they are adjacent. Now, if $|Z_2| \geq |Z|/3$, by choosing $Z'$ greedily, it follows that $Z_2$ contains a subset $Z'$ with the following specifications:
   \begin{itemize}
       \item $z_1, z_k \in Z'$; 
       \item $|Z'| \geq |Z_2|/(2y) \geq |Z|/(6y)$; and
       \item no two vertices in $Z'$ have a common neighbor in $H_i$. 
   \end{itemize}
   But then $(H_i, Z')$ is a triangular alignment, as desired. 

    Let $Z_3 = \{z_1, z_k\} \cup (Z \setminus (Z_1 \cup Z_2))$. From the previous two paragraphs, we may assume that $|Z_3| \geq |Z|/3$.   Let $R$ be a path from $z_1$ to $z_k$ with $R^* \subseteq H_{3-i}$,
   and let $H$ be the hole $z_1 \dd H_i \dd z_k \dd R \dd z_1$.
   If some $z \in Z \setminus \{z_1,z_k\}$ has at least four neighbors in $H_i$,
   then $(H,z)$ is a proper wheel in $G$, a contradiction. This proves that
   $|N_{H_i}(z)| \leq 3$ for every $z \in Z$.

   Let $z \in Z$. Define $Bad(z)=N_{H_i}[N_{H_i}(z)]$. Since $|N_{H_i}(z)| \leq 3$
   and $H_i$ is a path, it follows that    $|Bad(z)| \leq 9$.  
Since  $N_{Z}(h) <y$ for every $h \in H_i$,  we can greedily choose
   $Z' \subseteq Z$ with     $|Z'| \geq \frac{|Z_3|}{9y} \geq \frac{|Z|}{27y}$, $z_1,z_k \in Z'$ and  such that
   if $z, z' \in Z'$, then $z'$ is anticomplete to $Bad(z)$. 

   We claim that $(H_i,Z')$ is an alignment. Suppose not; then there exist $i < j < k$ such that $h_i, h_k \in N(z)$ and $h_j \in N(z')$ for $z, z' \in Z'$ with $z \neq z'$. We may assume that $i, j, k$ are chosen with $|k-i|$ minimum. It follows that $z$ has no neighbor in $\{h_{i+1}, \dots, h_{k-1}\}$. We consider three cases: 
   \begin{itemize}
       \item If $z'$ has a neighbor $h_l$ with $l > k$, then we define $P_1 = z$-$h_i$-$P$-$h_j$-$z'$ and $P_2 = z$-$h_k$-$P$-$h_l$-$z'$. 
       \item If $z'$ has a neighbor $h_l$ with $l < i$, then we define $P_1 = z$-$h_i$-$P$-$h_l$-$z'$ and $P_2 = z$-$h_k$-$P$-$h_j$-$z'$. 
       \item Otherwise, all neighbors of $z'$ are in $\{h_{i+1}, \dots, h_{k-1}\}$. Let $h_j$ and $h_l$ be the first and last neighbor of $z'$ in $h_{i+1}$-$\dots$-$h_{k-1}$, respectively. Then, since $z' \in Z_3$, it follows that $|l-j| > 1$. We define $P_1 = z$-$h_i$-$P$-$h_j$-$z'$ and $P_2 = z$-$h_k$-$P$-$h_l$-$z'$. 
   \end{itemize}
   Now we get a theta with ends $z,z'$ and paths
   $P_1$, $P_2$ and  a third path with interior in $H_{3-i}$, a contradiction. This proves that
   $(H_i,Z')$ is an alignment. Since $Z' \subseteq Z_3$, it follows that $(H_i,Z')$ is a wide alignment.
      This proves \eqref{align}.

\sta{\label{order}
There is a subset    $\hat Z \subseteq Z$ with $|\hat Z|\geq x$,  
  and a path $\hat{H_i} \subseteq H_i$ for $i=1,2$ such that
  $(\hat{H_i}, \hat{Z})$ is a consistent alignment or a connectifier.
  Moreover, if neither of $(\hat{H_1},\hat{Z)}, (\hat{H_2},\hat{Z})$ is a stellar connectifier, then 
  the order given on
$\hat{Z}$ by $(\hat{H_1},\hat{Z})$ and $(\hat{H_2}, \hat{Z})$
   is the same.}

Suppose first that some  $h \in H_1$ has at least $6x$ neighbors in $Z$.
Let $\hat{H_1}=\{h\}$, and let $\hat{Z} \subseteq Z \cap N(h)$
with $|\hat{Z}|=6x$.  If $(H_2,Z)$ is a connectifier,  
      then setting $\hat{H_2}=H_2$, we have that  \eqref{order} holds. So we may assume that $H_2$ is a path and every vertex of $Z$ has a neighbor in $H_2$. Let $\hat{H_2}$ be a minimal subpath of $H_2$ such that
      every $z \in \hat{Z}$ has a neighbor in $\hat{H_2}$. Then
      $H_2=h_1 \dd \dots \dd h_k$ and there exist $z_1,z_k \in \hat{Z}$
      such that $N_{\hat{H_2}}(z_i)=\{h_i\}$ for $i \in \{1,k\}$. Since for every $z \in  \hat{Z} \setminus \{z_1,z_k\}$, the graph
      $H_2 \cup \{z_1,z_k,z,h\}$ is not a theta and not a proper wheel with center
      $z$, it follows that every $z \in \hat{Z} \setminus \{z_1,z_k\}$ has exactly two neighbors in $H_2$
      and they are adjacent. Moreover, no vertex $x$ of $H_2$ has three or more neighbors in $\hat{Z}$, for otherwise $\{x, h\} \cup (N(x) \cap \hat{Z})$ contains a $K_{2,3}$. Now, by choosing $Z' \subseteq \hat{Z}$ greedily, we find a set $Z'$ with the following specifications:
   \begin{itemize}
       \item $z_1, z_k \in Z'$; 
       \item $|Z'| \geq \hat{Z}/6$; and
       \item no two vertices in $Z'$ have a common neighbor in $H_2$. 
   \end{itemize}
   But then $(H_1, Z')$ is a triangular alignment, and \eqref{order} holds. Therefore, each vertex $h \in H_1$ has at most $6x$ neighbors in $Z$; and by symmetry, it follows that each vertex $h \in H_2$ has at most $6x$ neighbors in $Z$. 

   % Next, we assume that $H_1$ is a subdivided star with root  $h$. 
   %    If $(H_2,Z)$ is a connectifier,
   %          then setting $\hat{H_2}=H_2$, we get that \eqref{order} holds.
   %    So we may assume that $H_2$ is a path and every vertex of $Z$ has a neighbor in $H_2$. By \eqref{align} (with $y = 6x$) and in view of the claim of the previous paragraph, there is a set $\hat{Z} \subseteq Z$ such that $(H_2,\hat{Z})$ is a consistent alignment
   %    and $|\hat{Z}| \geq x$. 
   %    Let
   %    $\hat{H_1}$ be the union of $h$ with the legs of $H_1$ containing
   %    neighbors of vertices in $\hat{Z}$; then \eqref{order} holds.

      In view of this, applying  \eqref{align} with $y = 6x$ (possibly twice), we conclude that there exists $Z'  \subseteq Z$ with      $|Z'|\geq x^2$ and subgraphs $H_i' \subseteq H_i$ such that
      $(H_i', Z')$ is a consistent alignment or a consistent connectifier.
         If none of them is a stellar connectifier, for $i \in \{1, 2\}$, let $\pi_i$ be the order given on $Z'$ by $(H_i,Z')$. By Theorem \ref{ESz}
   there exists $\hat{Z} \subseteq Z'$ such that (possibly reversing $H_i$)
   the orders $\pi_i$ restricted to $\hat{Z}$ are the same, as required.
   This proves \eqref{order}.

    \medskip
   
   Now Theorem \ref{twosides} follows from Lemma \ref{twosides_lemma}.
\end{proof}

\section{Bounding the number of non-hubs}
\label{sec:non-hubs}
In this section, we show that a bag of a structured tree decomposition of a
graph $G \in \mathcal{C}_{tt}$ contains a small number of vertices of
$G \setminus \Hub(G)$. This is the only place in the paper where the assumption that $G \in \mathcal{C}_{tt}$ (rather than $G \in \mathcal{C}_t$) is used.

  \begin{theorem}
    \label{adhesions}
    Let $G \in \mathcal{C}$  and let $(T,\chi)$ be a structured tree decomposition of $G$.
  Let $ v \in T$ and $Y \subseteq \chi(v)$
      be a stable set.
      Then there exist components $D_1, \dots, D_k$ of
    $G \setminus \chi(v)$ such that $Y \subseteq \bigcup_{i=1}^kN(D_i)$
      and $k \leq 4$.
      \end{theorem}

\begin{proof}
  Let $D_1$ be a component of $G \setminus \chi(v)$ such that $N(D_1) \cap Y$ is maximal.
  We may assume there exists $x_1 \in Y \setminus N(D_1)$.
  Since $(T, \chi)$ is structured and since $x_1$ is not
  complete to $\chi(v) \setminus \{x_1\}$, by Theorem \ref{thm:PMC_characterization} there exists a component $D_2$ of $G \setminus \chi(v)$   such that
  $x_1 \in N(D_2)$; choose $D_2$ with
  $N(D_2) \cap Y \cap N(D_1)$ is maximal. By the maximality
  of $N(D_1) \cap Y$, there exists $x_2 \in (Y \cap N(D_1)) \setminus N(D_2)$.
  Since $(T, \chi)$ is structured, Theorem \ref{thm:PMC_characterization}
  implies that there exists
a component $D_3$ of $G \setminus \chi(v)$ such that
$x_1,x_2 \in N(D_3)$.  By the choice of $D_2$, there exists $x_3 \in N(D_1) \cap N(D_2) \cap Y$ such that $x_3 \not \in N(D_3)$. For $\{i,j,k\}=\{1,2,3\}$, let
  $P_{ij}$ be a path from $x_i$ to $x_j$ with interior in $D_k$.

  Let $\mathcal{D}$ be the
  set of all components $D$ of $G \setminus \chi(v)$ that
  such that $|N(D) \cap \{x_1,x_2,x_3\}|>1$. Then $D_1,D_2,D_3 \in \mathcal{D}$.

  \sta{\label{eq:4} We have that  $|\mathcal{D} \setminus \{D_1,D_2,D_3\}| \leq 1$.}

    Suppose first that there is a component $D \in  \mathcal{D} \setminus \{D_1,D_2,D_3\}$ with $N(D) \cap \{x_1, x_2, x_3\} = \{x_1, x_2\}$. Then, we get a theta with ends $x_1, x_2$ and paths $x_1$-$P_{12}$-$x_2$, $x_1$-$P_{13}$-$x_3$-$P_{23}$-$x_2$ as well as a third path with interior in $D$ (using that $x_3 \not\in N(D)$). This is a contradiction, and proves that $N(D) = \{x_1, x_2, x_3\}$.
  
    Now suppose that $D, D' \in \mathcal{D} \setminus \{D_1, D_2, D_3\}$ with $D \neq D'$. Then we get a theta with ends $x_2,x_3$ and paths with interiors in
  $D_1,D, D'$, respectively, a contradiction. This proves \eqref{eq:4}.

  \medskip
  Thus $|\mathcal{D}| \leq 4$.
  We may assume that there is a vertex $x \in Y$ such that
  $x \not \in \bigcup_{D \in \mathcal{D}}N(D)$.
  Since  $(T, \chi)$ is structured, Theorem \ref{thm:PMC_characterization} implies that
  there exist paths
  $P_1,P_2$ where $P_i$ is from $x$ to
  $x_i$, and $P_i^* \cap \chi(v)=\emptyset$. It follows that  each $P_i^*$
  is contained in a component $F_i$ of $G \setminus \chi(v)$.
  Since $F_i \not \in \mathcal{D}$, it follows that $F_1 \neq F_2$,
  $x_2,x_3 \not \in N(F_1)$, and $x_1,x_3 \not \in N(F_2)$.
  Now we get a theta with ends $x_1,x_2$ and paths
  $x_1 \dd P_{13} \dd x_3 \dd P_{23} \dd x_2$, $x_1 \dd P_{12} \dd x_2$
  and $x_1 \dd P_1 \dd x \dd P_2 \dd x_2$, a contradiction. This
  proves Theorem~\ref{adhesions}.
  \end{proof}

Next we show:
\begin{theorem} \label{smallminimal}
  Let $G \in \mathcal{C}_{tt}$, let $S$ be a minimal separator of $G$,
  and let $Y \subseteq S \setminus \Hub(G)$ be stable. Let $\tau=\tau(t)$
  be as in Theorem \ref{twosides}. Then 
  $|Y| \leq \tau$.
\end{theorem}

\begin{proof}
Suppose $|Y| \geq \tau$. 
Let $D_1,D_2$ be distinct full components for $X$.
Apply Theorem \ref{twosides} to $D_1 \cup D_2 \cup Y$, and let
$H_1,H_2,X$ be as in the conclusion of \ref{twosides}. Then a routine case analysis (whose details we leave to the reader) shows that
$H_1 \cup H_2 \cup X$ contains a generalized $t$-pyramid in $G$, a contradiction.
This proves Theorem \ref{smallminimal}.
    \end{proof}

We can now prove the main result of this section.

\begin{theorem} \label{boundhubs}
  Let $G \in \mathcal{C}_{tt}$, let $(T, \chi)$ be a structured tree decomposition
  of $G$, and let $v \in T$.
  Let $Y \subseteq \chi(v) \setminus \Hub(G)$ be stable.
  Let $\tau=\tau(t)$ be as in Theorem \ref{twosides}.
  Then $|Y| \leq 4 \tau$.   
\end{theorem}

\begin{proof}
  Suppose $|Y| > 4 \tau$. By Theorem \ref{adhesions},  there exists
  a component $D$ of $G \setminus \chi(v)$ such that
  $|N(D) \cap Y| > \tau$. By Theorem \ref{prop:PMC_adhesions_are_seps}, 
  $N(D) \cap Y$ is a minimal separator of $G$, contrary to Theorem \ref{smallminimal}.
  This proves Theorem \ref{boundhubs}.
\end{proof}

\section{Putting everything together} \label{sec:proof}

For the remainder of the paper, all logarithms are taken in base 2.
We start with the following theorem from \cite{TWIII}:

\begin{theorem}[Abrishami, Chudnovsky, Hajebi, Spirkl \cite{TWIII}]
  \label{logncollections}
  Let $t \in \mathbb{N}$, and let $G$ be (theta, $K_t$)-free with $|V(G)|=n$.
  There exist an integer $d=d(t)$ be 
  and   a partition
  $(S_1, \dots, S_k)$   of $V(G)$ with the following properties:
  \begin{enumerate}
  \item $k \leq  \frac{d}{4} \log n$.
  \item $S_i$ is a stable set for every $i \in \{1, \dots, k\}$.
  \item For every $i \in \{1, \dots, k\}$ and $v \in S_i$ we have
    $\deg_{G \setminus \bigcup_{j <i}S_j}(v)  \leq d$. \label{hubsequence-3}
  \end{enumerate}
  \end{theorem}

Let $G \in \mathcal{C}_{tt}$ be a graph. A {\em hub-partition} of $G$ is a partition
$S_1, \dots, S_k$ of $\Hub(G)$ as in
Theorem \ref{logncollections};
we call $k$ the {\em order} of the partition.
We call the {\em hub-dimension} of $G$ (denoting it by
$\hdim(G)$) the smallest $k$ such that $G$ as a hub-partition of order $k$.

For the remainder of this section, let us fix $t \in \mathbb{N}$. Let $d = d(t)$ as in Theorem \ref{logncollections}. Let $c_t=\gamma(t)+1$ with $\gamma(t)$ as in Theorem \ref{noblocksmalltw_Ct}.
Let $k_t=k(t)$ be as in Theorem \ref{banana}. Let $m=k_t+2d$.
Let $\Psi=4 \tau(t)$ where $\tau(t)$ is as in Theorem \ref{twosides}.
Let $\Delta=(2m-1)m+c_t$.

The following is a strengthening of Theorem \ref{main}, which we prove by induction on
$\hdim(G)$. By  Theorem \ref{logncollections}, $\hdim(G) \leq \frac{d}{4} \log n$ for every $G \in \mathcal{C}_{tt}$, so  Theorem \ref{diminduction}
immediately implies  Theorem \ref{main}.

\begin{theorem}
  \label{diminduction}
  Let $G \in \mathcal{C}_{tt}$ with $|V(G)|=n$. 
  Then  $\tw(G) \leq c_t + \Delta \Psi (\log n+ \hdim(G))$.
\end{theorem}
\begin{proof}
  The proof is by induction on $\hdim(G)$, and  for a fixed value of $\hdim$, by induction on $n$.
  If $hdim(G)=0$, then by Theorem \ref{noblocksmalltw_Ct}, we have that $\tw(G) \leq c_t$ and
  Theorem \ref{diminduction} holds as required.
  Thus we may assume $\hdim(G)>0$.
  A special case of Lemma 3.1 from \cite{cliquetw} shows that clique cutsets do not affect treewidth; thus we may assume that $G$ does not admit a clique cutset.

  \sta{\label{smallsep} If $G$ has a balanced separator of size
    at most $m$, then the theorem holds.}

  Let $X$ be a balanced separator of $G$ of size at most $m$. Let $D_1, \dots, D_s$ be the components of $G \setminus X$. Since $|D_i| \leq \frac{n}{2}$
  for every $i \in \{1, \dots, s\}$, it follows from our induction on $n$ that
  $\tw(D_i) \leq c_t + \Delta \Psi (\log n+ \hdim(G)-1)$. Then, by Lemma \ref{lem:septotd}, the treewidth of $G$ is at most
  \begin{align*}
     & \ \ \ \ \max_{i \in \{1, \dots, s\} }\tw(D_i) + |X|\\
  &\leq  c_t + \Delta   \Psi (\log n+ \hdim(G)-1)+ m\\
&\leq   c_t + \Delta  \Psi(\log n+ \hdim(G)), 
  \end{align*}
  where we use that $\Delta \geq m$. 
  This proves \eqref{smallsep}.

  \medskip
  
In view of \eqref{smallsep}, we may assume that $G$ does not admit a balanced separator of size at most $m$, and so the results of Section~\ref{sec:centralbag}
apply.

Let $S_1, \dots, S_k$ be a hub-partition of $G$ with $k=\hdim(G)$.
  We now use the terminology from Section \ref{sec:centralbag}.
If follows from the definition of $S_1$ that every vertex in $S_1$
is $d$-safe.
Let $(T,\chi)$ be an $m$-atomic
tree decomposition of $G$. Let $t_0$ be a  center for $T$,
and let $\beta=\beta(S_1)$ be as in Section \ref{sec:centralbag}. 
Write $\beta_0=\chi(t_0)$.
Our first goal is to  prove:

\sta{\label{betagood} There is  a tree decomposition
  $(T_{\beta}, \chi_{\beta})$ of $\beta$ of
  such that for every $t \in T_{\beta}$, we have that $|\chi_{\beta}(t) \cap \beta_0| \leq c_t + \Delta \Psi (\log n+ k-1)+(\Delta - m) \Psi$.}

We start with:

\sta{\label{smallsepbeta}
If $\beta$ has a balanced separator  of size at most
$2m(m-1)+c_t$, then \eqref{betagood} holds.}

To prove \eqref{smallsepbeta}, we let $X$ be a balanced separator of size at most $2m(m-1)+c_t = \Delta - m$ of
$\beta$.
%Then  $|\beta| < \frac{n}{2}$.
Let $D_1, \dots, D_s$ be the components of $\beta \setminus X$. Since $|D_i| \leq \frac{|\beta|}{2} \leq \frac{n}{2}$
  for every $i \in \{1, \dots, s\}$, it follows from our induction on $n$ that
  $\tw(D_i) \leq c_t + \Delta  \Psi (\log n+ \hdim(G)-1)$. Therefore, by Lemma \ref{lem:septotd} applied to $\beta$ and $X$, we conclude that: 
  \begin{align*}
      \tw(\beta) &\leq \max_{i \in \{1, \dots, s\} }\tw(D_i) + |X| \\
      &\leq  c_t + \Delta   \Psi (\log |\beta|+ \hdim(G)-1) + \Delta-m\\
      & \leq c_t + \Delta  \Psi(\log n+ \hdim(G)-1)+  (\Delta-m) \Psi.
  \end{align*}
  This proves \eqref{smallsepbeta}.

  \medskip
   Now we may assume that $\beta$ does not have a balanced separator  of size at most $2m(m-1)+c_t$. Therefore 
   $\beta^A(S_1)$ is defined, 
   as in Section \ref{sec:centralbag}.  By Theorem \ref{A_centralbag}\eqref{A-4}, we have that   $S_1 \cap \Hub(\beta^A(S_1))=\emptyset$ and
   $S_2 \cap Hub(\beta^A(S_1)), \dots, S_k \cap Hub(\beta^A(S_1))$ is a hub-partition of
   $\beta^A(S_1)$. 
It follows that $\hdim(\beta^A(S_1)) \leq k-1$.

Let $D_1, \dots D_s$ be
the components of $\beta \setminus \beta^A(S_1)$. By Theorem \ref{A_centralbag}\eqref{A-3}, we have that 
$|D_i| \leq \frac{n}{2}$ for every $i \in \{1, \dots, s\}$. Moreover, by induction on $k$, we obtain that
$$\tw(\beta^A(S_1))    \leq c_t +\Delta   \Psi (\log n+ k-1).$$
Our induction on $n$ further implies that
$$\tw(D_i)   \leq c_t +  \Delta  \Psi (\log n+ k-1).$$

By Theorem \ref{structured}, the graph $\beta^A(S_1)$ admits a structured  tree decomposition $(T_0, \chi_0)$ of
width $\tw(\beta^A(S_1))$. For every $i \in \{1, \dots, s\}$, let  $(T_i, \chi_i)$ be a  tree decomposition of $D_i$ of width $\tw(D_i)$. Let $(T_\beta,\chi_\beta)$ be a tree decomposition of $\beta$ obtained as in Theorem \ref{A_extendtree}. We claim:

 \sta{\label{betasmall}
  For every $i \in \{0, \dots, s\}$ and for every $t \in T_i$, we have that 
  $|(\chi_{\beta}(t) \setminus \chi_i(t)) \cap \beta_0| \leq 2m(m-1)\Psi$.}

Let $i \in \{1, \dots, s\}$ and let $t \in T_i$.
It follows immediately from Theorem \ref{A_extendtree} that
$|(\chi_{\beta}(t) \setminus \chi_i(t)) \cap \beta_0| \leq 2m(m-1)$.
Now let $t \in T_0$. Since     $S_1 \cap \Hub(\beta^A(S_1))=\emptyset$,
we deduce from Theorem \ref{boundhubs} applied to $\beta^A(S_1)$ that 
 $|\chi_0(t) \cap \Core(S_1)| \leq \Psi $. Now again
it follows from Theorem \ref{A_extendtree} that
$|(\chi_\beta(t) \setminus \chi_0(t)) \cap \beta_0| \leq 2m(m-1) \Psi$.
This proves \eqref{betasmall}.

\medskip
Now \eqref{betagood} follows immediately from \eqref{betasmall}, using that $\Delta-m = 2m(m-1) + c_t \geq 2m(m-1)$.

\medskip
Next we use Theorem \ref{delta_extendtree} to turn $(T_\beta, \chi_\beta)$ into a
tree decomposition of $G$ of the required  width. Let $D_1', \dots, D_r'$ be
the components of $G \setminus \beta_0$.
In view of Theorem \ref{structured} and \eqref{betagood}, we let
 $(T_0', \chi_0')$ be a structured tree decomposition
of $\beta$ such that for every $t \in T_0'$, we have that
$|\chi_0'(t) \cap \beta_0| \leq c_t + \Delta \Psi (\log n+ k-1)+
(\Delta-m)\Psi$.

Let $i \in \{1, \dots, r\}$.
Since $t_0$ is a center of $T$, it follows that 
$|D_i'| \leq \frac{n}{2}$ for every $i \in \{1, \dots, r\}$.
By induction on $n$, we have that
$$\tw(D_i')   \leq c_t + \Delta\Psi(\log n+ k-1).$$
Let  $(T_i', \chi_i')$ be a tree decomposition of $D_i'$ of width $\tw(D_i')$.
Let $(U,\psi)$ be a tree decomposition of $G$ obtained as in Theorem \ref{delta_extendtree}. Recall that for $u \in U$, $\psi(u)$ is defined as follows.
\begin{itemize}
    \item  If $u \in V(T_0')$, let
      $$\psi(u)=(S_1)_{bad} \cup (\chi_0'(u)  \cap \beta_0) \cup
      \bigcup_{t_i \text{ s.t. }   \chi_0'(u) \cap (Conn(t_i) \setminus \chi(t_0)) \neq \emptyset} (\beta_0 \cap \chi(t_i)).$$
      
    \item If $u \in V(T_i')$ for $i \in \{1, \dots, r\}$, let
    $$\psi(u)=(S_{1})_{bad} \cup \chi_i'(u) \cup (\beta_0 \cap \chi(t_i)).$$
\end{itemize} We claim that 
$$\width(U, \psi) \leq  c_t + \Delta \Psi(\log n+ k).$$
  To prove this claim, we let $u \in V(U)$; we will establish an upper bound on $|\psi(u)|$. By Lemma \ref{noncoop}, we have $|(S_1)_{bad}| \leq 1$. Now let $i \in \{0, \dots, r\}$ be such that $u \in T_i'$.

  Suppose first that  $i>0$. Since $|\beta_0 \cap \chi(t_i)| < m$ (because $(T, \chi)$ is $m$-atomic and therefore $m$-lean by Theorem \ref{atomictolean}), it follows that
$$|\psi(u)| \leq 1+|\chi_i'(u)| + (m-1) \leq \tw(D_i') + m \leq c_t +\Delta \Psi(\log n+ k),$$
  using that $m \leq \Delta$, as required.
  
  Thus we may assume that $i=0$. By Theorem \ref{connectors}\eqref{C-3}, we have that $(\chi_0'(u) \setminus \beta_0) \cap \Hub(\beta)=\emptyset$. Since for distinct $t_1,t_2 \in N_T(t_0)$, the set
      $Conn(t_1) \setminus \beta_0$ is disjoint from and anticomplete to the set
      $Conn(t_2) \setminus \beta_0$, applying Theorem \ref{boundhubs} implies that
      $$|\{t_i : \chi_0'(u) \cap (Conn(t_i) \setminus \beta_0) \neq \emptyset\}| \leq \Psi.$$
Since    by \eqref{betagood}, we have 
$|\chi_0'(u) \cap \beta_0| \leq   c_t + \Delta  \Psi (\log n+ k-1)+
\Psi(\Delta-m)$,
and since $|\beta_0 \cap \chi(t_i)| < m$ for every $i$, 
  we deduce that
  $$|\psi(u)| \leq  1+ c_t + \Delta \Psi (\log n+ k-1)+
  (\Delta-1)  \Psi  \leq 
 c_t +\Delta \Psi (\log n+ k),$$
 as required.
  This completes the proof of  \ref{diminduction}.
    \end{proof}

\section{Algorithmic consequences}\label{algsec}

We now repeat the main points of the  last section of \cite{TWIII} to explain
the algorithmic significance of Theorem \ref{main}. We need the following theorem from \cite{TWIII}:

\begin{theorem}[Abrishami, Chudnovsky, Hajebi, Spirkl \cite{TWIII}]
  \label{generalalg}
Let {\rm\texttt{P}} be a problem which admits an algorithm running in time $\mathcal{O}(2^{\mathcal{O}(k)}|V(G)|)$ on graphs $G$ with a given tree decomposition of width at most $k$. Also, let $\mathcal{G}$ be a class of graphs for which there exists a constant $c=c(\mathcal{G})$ such that $\tw(G)\leq c\log(|V(G)|)$ for all $G\in \mathcal{G}$. Then {\rm\texttt{P}} is polynomial-time solvable in $\mathcal{G}$.
\end{theorem}

In view of Theorem \ref{main}, Theorem \ref{generalalg} and results of
\cite{Bodlaender1988DynamicTreewidth}, we conclude the following.

\begin{theorem}\label{algfinal}
Let $t\geq 1$ be fixed and {\rm\texttt{P}} be a problem which admits an algorithm running in time $\mathcal{O}(2^{\mathcal{O}(k)}|V(G)|)$ on graphs $G$ with a given tree decomposition of width at most $k$. Then {\rm\texttt{P}} is polynomial-time solvable in $\mathcal{C}_{tt}$. In particular, \textsc{Stable Set}, \textsc{Vertex Cover}, \textsc{Dominating Set} and \textsc{$r$-Coloring} (with fixed $r$) are all polynomial-time solvable in $\mathcal{C}_{tt}$.
\end{theorem}

Let us now discuss another important problem, and that is  \textsc{Coloring}.
It is well-known (and also follows immediately from
 Theorem \ref{logncollections}), that for every $t$ there exists a number
 $d(t)$ such that  every graph  in $\mathcal{C}_t$ has a vertex of degree
 at most  $d(t)$. This implies that all graphs in $\mathcal{C}_t$ 
  have chromatic number at most $d(t)$+1. Also, for each fixed $r$, by Theorem  \ref{algfinal}, $r$-\textsc{Coloring} is polynomial-time solvable in $\mathcal{C}_{tt}$. Now by solving  $r$-\textsc{Coloring} for every $r \in \{1, \dots, d(t)+1\}$, we obtain a polynomial-time algorithm for \textsc{Coloring} in $\mathcal{C}_{tt}$.

\section{Acknowledgments}
We are grateful to Rose McCarty for many helpful discussions.
We thank Reinhard Diestel for being an invaluable source of information about
tree decompositions, and for his generosity with his time. We thank Irena Penev and Kristina  Vu\v{s}kovi\'c for
telling us about important algorithmic problems whose complexity was open for
the class of even-hole free graphs.

\end{document}